\documentclass[10pt]{amsart}
\usepackage{a4}
\usepackage[english]{babel}
\usepackage{amsmath,amsthm,amssymb}
\usepackage{amsaddr}
\usepackage{bm}
\usepackage{graphicx}
\usepackage{epsfig,epstopdf}
\usepackage{todonotes}

\newtheorem{thm}{Theorem}[section]
\newtheorem{lem}[thm]{Lemma}
\newtheorem{cor}[thm]{Corollary}
\newtheorem{rem}[thm]{Remark}

\newtheorem{defn}[thm]{Definition}

\newcommand{\R}{{\mathbb R}}

\newcommand{\C}{{\mathbb C}}

\newcommand{\cO}{{\mathcal O}}

\newcommand{\cC}{{\mathcal C}}

\newcommand{\one}{{\mathbf 1}}
\newcommand{\zero}{{\mathbf 0}}

\newcommand{\eps}{\epsilon}

\newif\ifmatlab\matlabtrue
\newcommand{\matlab}{MATLAB\ifmatlab\textsuperscript{\textregistered}
\matlabfalse\fi}
\newcommand{\dd}{\, \mathrm{d}}

\numberwithin{equation}{section}
\begin{document}

\title[Model reduction of Fokker--Planck and quantum Liouville equations]{Model 
reduction of controlled Fokker--Planck and Liouville--von Neumann equations}

\author[P. Benner]{Peter Benner}
\address{Computational Methods in Systems and Control Theory\\ 
Max Planck Institute for Dynamics of Complex Technical Systems\\ Sandtorstr. 1, 
D-39106 Magdeburg, Germany}
\email{benner@mpi-magdeburg.mpg.de}

\author[T. Breiten]{Tobias Breiten}
\address{Institut f\"ur Mathematik, Karl-Franzens-Universit\"at\\ Heinrichstr. 
36/III, A-8010 Graz, Austria}
\email{tobias.breiten@uni-graz.at}

\author[C. Hartmann]{Carsten Hartmann}
\address{Institut f\"ur Mathematik, Brandenburgische Technische 
Universit\"at\\Konrad-Wachsmann-Allee 1, D-03046 Cottbus, Germany}
\email{carsten.hartmann@b-tu.de}

\author[B. Schmidt]{Burkhard Schmidt}
\address{Institut f\"ur Mathematik, Freie Universit\"at Berlin\\Arnimallee 6, 
D-14195 Berlin, Germany}
\email{burkhard.schmidt@fu-berlin.de}

\begin{abstract}
  Model reduction methods for bilinear control systems are compared by means 
  of practical examples of Liouville--von Neumann and Fokker--Planck type. 
Methods based on balancing generalized system Gramians and on minimizing an 
$\mathcal{H}_2$-type cost functional are considered. The focus is on  
  the numerical implementation and a thorough comparison of the 
  methods. Structure and stability preservation are investigated, and the 
competitiveness of the approaches is shown for practically relevant, 
large-scale examples.
\end{abstract}

\date{\today}

\maketitle

\tableofcontents

\section{Introduction}

Due to the growing ability to accurately manipulate single molecules by 
spectroscopic techniques, numerical methods for the control of molecular 
systems 
have recently attracted a lot of attention 
\cite{Borzi2013,Hartmann2012,Moerner2015,Zhang2014}. Key applications involve probing of 
mechanical properties of biomolecules by force microscopy and optical tweezers 
\cite{Souza2012,Hummer2010}, or the control of chemical reaction dynamics by 
temporally shaped femtosecond laser pulses in femtochemistry 
\cite{Tiwari2008,Zewail1997}.
A key feature of these small systems is that they are open systems, in that 
they are subject to noise and dissipation induced by the interaction with their 
environment, as a consequence of which the dynamics are inherently random and 
the description is on the level of probability distributions or measures rather 
than trajectories \cite{Rey-Bellet2006}.

Depending on whether or not quantum effects play a role, the evolution of the 
corresponding probability distributions is governed by parabolic partial 
differential equations of either Liouville--von Neumann or Fokker--Planck type. 
The fact that the dynamics are controlled implies that the equations are 
bilinear as the control acts as an advection term that is coupled linearly to 
the probability distribution, but the main computational bottleneck clearly is 
that the equations, in spatially semi-discretized form, are high-dimensional 
which explains why model reduction is an issue; for example, in catalysis, 
optimal shaping of laser pulses requires the iterated integration of the 
dissipative Liouville--von Neumann (LvN) equation for reduced quantum 
mechanical 
density matrices, the spatial dimension of which grows quadratically with the 
number of quantum states involved \cite{breuer1997}; cf.~\cite{LeBris2002}.

Many nonlinear control systems can be represented as bilinear systems by a 
suitable change of coordinates (as well as linear parametric systems), and it 
therefore does not come as a surprise that model reduction of bilinear control 
systems has recently been a field of intense research; see 
\cite{Baur2014,Willcox2015} and the references therein. In recent years, 
various 
model reduction techniques that were only available for linear systems have 
been 
extended to the bilinear case, among which are Krylov subspace techniques 
\cite{nice2016,bai2006,breiten2010,lin2009,phillips2003}, interpolation-based 
approaches \cite{ahmad2017,Benner2012,flagg2012,Flagg2015}, balanced model 
reduction \cite{albaiyat1993,benner2011,Boris2011,hartmann2013}, empirical POD 
\cite{condon2004,condon2005,lall2002}, or $\mathcal{H}_2$-optimal model 
reduction \cite{Benner2012,Flagg2015,zhang2002}. 
The downside of many available methods is their lack of structure preservation, 
most importantly, regarding asymptotic stability. In our case, positivity is an 
issue too, as we are dealing with probability distributions.  

In this paper we compare two different model reduction techniques that 
represent 
different philosophies of model order reduction, with the focus being on 
practical computations and numerical tests rather than a theoretical analysis. 
The first approach is based on the interpolation of the Volterra series 
representation of the system's transfer function and gives a local 
$\mathcal{H}_2$-optimal approximation, because the interpolation is chosen so 
that the system satisfies the necessary $\mathcal{H}_2$-optimality conditions 
upon convergence of the algorithm; see \cite{Benner2012} for details. The 
second approach is based on balancing the controllable and observable 
subspaces, 
 and exploits the properties of the underlying dynamical system in that it uses 
the properties of the controllability and observability Gramians to identify 
suitable small parameters that are sent to 0 to yield a reduced-order system; 
for details, we refer to \cite{hartmann2013}. 
Both methods require the solution of large-scale matrix Sylvester or Lyapunov 
equations. While the computational effort of balanced model reduction is 
essentially determined by the solution of two generalized Lyapunov equations 
for 
controllability and observability Gramians, the effort of the 
$\mathcal{H}_2$-optimal interpolation method is mainly due to the solution of 
two generalized Sylvester equations in each step of the bilinear iterative 
rational Krylov algorithm (B-IRKA). We stress that both generalized Lyapunov or 
Sylvester equations can be solved iteratively at comparable numerical cost (for 
a given accuracy), but they all require the dynamics of the uncontrolled system 
to be asymptotically stable \cite{Wachspress1988}. However, as both the 
dissipative LvN and Fokker-Planck operators have a simple eigenvalue zero, 
stability has to be enforced before solving Lyapunov or Sylvester equations, 
and 
in this paper we systematically compared stabilization techniques for both 
approaches. 

The outline of the article is as follows: In Section \ref{sec:bilin} we briefly 
discuss the basic properties of bilinear systems and set the notation for the 
remainder of the article. Model reduction by $\mathcal{H}_2$-norm minimization 
and balancing are reviewed in Sections \ref{sec:h2} and \ref{sec:balancing}, 
along with some details regarding the numerical implementation for the specific 
applications considered in this paper in Section \ref{sec:numeric}.  Finally, 
in 
Section \ref{sec:fpe} we study model reduction of the Fokker-Planck equation 
comparing balancing and $\mathcal{H}_2$-norm minimization, and in Section 
\ref{sec:lvne} we carry out a similar study for the dissipative Liouville--von 
Neumann equation. We discuss our observations in Section \ref{sec:conclusion}. 
The article contains an appendix, Appendix \ref{sec:stability}, that records 
some technical lemmas related to the asymptotic stability of bilinear systems.

\section{Bilinear control systems}\label{sec:bilin}

We start by setting the notation that will be used throughout this article. Let 
$x(t)\in\C^{n}$ be  governed  by the time-inhomogeneous differential equation
\begin{equation}\label{bilin1}
  \frac{\dd x}{\dd t} = A x + \sum_{k=1}^{m} \left(N_k x  + 
b_{k}\right)u_{k}\,,\quad x(0)=x_{0}\,,
 \end{equation}
with coefficients $A, N_{k}\in\C^{n\times n}$, $b_{k}\in\C^{n}$ and 
$u=(u_{1},\ldots, u_{m})^T$ being a vector of bounded measurable controls 
$u_{i}(t)\in U\subset\C$. We assume that not all state variables $x$ are 
relevant or observable, so we augment  (\ref{bilin1}) by a linear output 
equation 
\begin{equation}\label{bilin2}
  y =  C x\,, 
 \end{equation}
with $C\in\C^{l\times n}$, $l\le n$. The systems of equations 
(\ref{bilin1})--(\ref{bilin2}) is called a \emph{bilinear control system} with 
inputs $u(t)\in U^{m}\subset\C^{m}$ and outputs $y(t)\in\C^{l}$.

As is well-known, see e.g. \cite{Rug82,zhang2002}, an explicit output 
representation for \eqref{bilin2} can be obtained by means of successive 
approximations. The resulting so-called \emph{Volterra series} is given as
\begin{equation}\label{volterra}
\begin{aligned}
  y(t) &= \sum_{k=1}^\infty  \int_0^{\infty} \cdots \int_0^{\infty} 
\sum_{\ell_1,\dots,\ell_k=1}^m C e^{As_k} N_{\ell_1} e^{As_{k-1}} N_{\ell_2} 
\cdots e^{As_2}N_{\ell_{k-1}}e^{As_1} b_{\ell_k} \\
&\qquad \qquad \times u_{\ell_1}(t-s_k)u_{\ell_2}(t-s_k-s_{k-1}) \cdots 
u_{\ell_k}(t- \sum_{j=1}^k s_j) \, \dd s_1\cdots \dd s_k.
\end{aligned}
\end{equation}
Moreover, based on a multivariate Laplace transform of these integrands, the 
system can alternatively be analyzed in a generalized frequency domain by means 
of generalized transfer functions. Since this will not be essential for the 
results presented here, we refrain from a more detailed discussion and refer 
to, e.g., \cite{Rug82}.

\subsection{Reduced-order models}
\label{sec:reduced}

We seek coefficients $\hat{A},\hat{N}_{k}\in\C^{d\times d}$, 
$\hat{b}_{k}\in\C^{d}$ and $\hat{C}\in\C^{l\times d}$ with $d\ll n$ such that 	
\begin{equation}\label{bilinred}
\begin{aligned}
	\frac{\dd\xi}{\dd t} & =\hat{A}\xi + \sum_{k=1}^{m} 
\left(\hat{N}_{k}\xi + \hat{b}_{k}\right)u_{k}\,,\quad \xi(0) = \xi_{0}\,,\\
	\hat{y} & = \hat{C}\xi
  \end{aligned}
 \end{equation}
has an input-output behavior that is similar to (\ref{bilin1})--(\ref{bilin2}). 
In other  words, we seek a reduced-order model with the property that for any 
admissible control input $u$ (to be defined below), the error in the output 
signal,
\begin{equation}
\delta(t) = \|\hat{y}(t)-y(t)\|\,,
\end{equation}
is small, relative to $\|u\|$ (in some norm) and uniformly on bounded time 
intervals.


As will be outlined below, both model reduction schemes considered in this paper 
are closely related to the 
solutions of the following adjoint pair of generalized Lyapunov equations:  
\begin{equation}\label{P}
AP  + P A^{*} + \sum_{k=1}^{m}N_k P N_k^{*} + BB^* = 0\,
\end{equation}
and 
\begin{equation}\label{Q}
A^{*}Q + Q A + \sum_{k=1}^{m}N^{*}_k Q N_k + C^*C = 0\,,
\end{equation}
where, in the first equation, we have introduced the shorthand 
$B=(b_{1},\ldots,b_{m})\in\C^{n\times m}$.  
The Hermitian and positive 
semi-definite matrices $P,Q\in\C^{n\times n}$ are called the 
\emph{controllability and observability Gramians associated with 
(\ref{bilin1})--(\ref{bilin2})}---assuming well-posedness of the Lyapunov 
equations and hence existence and uniqueness of $P$ and $Q$. The relevance of 
the Gramians for model reduction is related to the fact that the nullspace of 
the controllability Gramian contains only states that cannot be reached by any 
bounded measurable control and that the system will not produce any output 
signal, if the dynamics is initialized in the nullspace of the observability 
Gramian \cite{isidori1973}; as a consequence one can eliminate states that 
belong to $\ker(P)\cap\ker(Q)$ without affecting the input-output behavior of 
(\ref{bilin1})--(\ref{bilin2}); cf.~\cite{albaiyat1993}.

\subsection{Standing assumptions}
\label{sec:assumptions}
The following assumptions will be used throughout to guarantee existence and 
uniqueness of the solutions to the generalized Lyapunov equations (Assumption 
1) 
and existence and uniqueness of the solution of the bilinear system 
(\ref{bilin1}) for all $t\ge 0$ (Assumptions 2 and 3):\label{ass}\\

\textbf{Assumption 1:} There exists constants $\lambda,\mu>0$, such that 
\[
\|\exp(A t)\|\le \lambda\exp(-\mu t)
\] 
and 
\[
\frac{\lambda^{2}}{2\mu}\sum_{k=1}^{m} \| N_{k}\|^{2} < 1\,, 
\]
where $\|\cdot\|=\|\cdot\|_{2}$ is the matrix 2-norm that is induced by the 
Euclidean norm $|\cdot|$.\\

\textbf{Assumption 2:} The bilinear system (\ref{bilin1})--(\ref{bilin2}) is 
bounded-input-bounded-output (BIBO) stable, i.e., there exists $M<\infty$, such 
that for any input $u$ with 
\[
\|u\|_{\infty} = \sup_{t\in[0,\infty)}|u(t)| \le   M 
\]
the output $y(t)$ is uniformly bounded.\\

\textbf{Assumption 3:} The admissible controls $u\colon[0,\infty)\to 
U^{m}\subset\C^{m}$ are continuous, bounded and square integrable, i.e., $u\in 
\cC_{b}([0,\infty),U^{m})$ with 

\[
\|u\|_{2} = \left(\int_{0}^{\infty} |u(t)|^{2}\, \dd t \right)^{1/2} <\infty .
\]
Specifically, we require that the admissible controls are uniformly bounded by 
\[
M < \frac{\mu}{\lambda}\sum_{k=1}^{m} \| N_{k}\|\,, 
\]
with $\lambda,\mu$ as in Assumption 1, which by BIBO stability (Assumptions 2) 
implies that the output $y(t)$ is bounded for all $t\ge 0$ 
(cf.~\cite{siu1991}).


\section{$\mathcal{H}_{2}$ optimal model reduction of bilinear 
systems}\label{sec:h2}

In this section, we recall some existing results on 
$\mathcal{H}_2$-optimal model order reduction for bilinear systems. For 
a more detailed presentation, see \cite{zhang2002,Benner2012,flagg2012}. 

For a better understanding of the subsequent concepts, let us briefly focus on
the linear case, i.e., $N_k=0$ in \eqref{bilin1}. Here, the Volterra series 
representation \eqref{volterra} simplifies to $y(t) = \int_0^\infty 
Ce^{As}Bu(t-s)  \dd s.$ If the input signal is a Dirac mass at $0,$ we 
obtain the \emph{impulse response} $h(t)=Ce^{At}B.$ The $\mathcal{H}_2$-norm 
for linear systems now is simply defined as the $L_2$-norm of the impulse 
response, i.e.,
\begin{align*}
  \| h\|_{L^2(0,\infty;\C^m)}^2=\int_0^\infty 
\mathrm{tr}(B^*e^{A^*t}C^*Ce^{At}B) \dd t.
\end{align*}

Based on the latter definition and the Volterra series, in  
\cite{zhang2002}, the $\mathcal{H}_2$-norm has been generalized for bilinear 
systems as follows.
\begin{defn}
Let $\Sigma=(A,N_1,\dots,N_k,B,C)$ denote a bilinear system as in 
\eqref{bilin1}. We then define its $\mathcal{H}_2$-norm by
  $$\|\Sigma \|_{\mathcal{H}_2}^2 = \mathrm{tr} \left(\sum_{k=1}^\infty \int 
_0^\infty \cdots \int_0^\infty \sum_{\ell_1,\dots,\ell_k=1}^m 
g_k^{(\ell_1,\dots,\ell_k)} (g_k^{(\ell_1,\dots,\ell_k)})^* \dd s_1\cdots \dd 
s_k 
\right),$$ 
with $ g_k^{(\ell_1,\dots,\ell_k)} (s_1,\dots,s_k) = Ce^{As_k}N_{\ell_1} 
e^{As_{k-1}} N_{\ell_2} \cdots e^{As_1} b_{\ell_k}.$
\end{defn}
Obviously, for a bilinear system having a finite $\mathcal{H}_2$-norm, it is 
required that the system is stable in the linear sense, i.e., $A$ has only 
eigenvalues in $\C_-.$ Moreover, the matrices $N_k$ have to be sufficiently 
bounded. From \cite{zhang2002}, let us recall that Assumption 1 ensures that 
the bilinear system under consideration has a finite $\mathcal{H}_2$-norm, 
which, moreover, can be computed by means of the solution $P$ and $Q$ of the 
generalized Lyapunov equations \eqref{P} and \eqref{Q}, respectively. In 
particular, we have that
\begin{align*}
 \| \Sigma\|_{\mathcal{H}_2}^2 = \mathrm{tr}(CPC^*) = \mathrm{tr}(B^*QB).
\end{align*}
Given a fixed system dimension $l$ 
the goal of $\mathcal{H}_2$-optimal model order reduction now is to construct a 
reduced-order bilinear system $\tilde{\Sigma}$ such that 
$$\|\Sigma-\tilde{\Sigma} \|_{\mathcal{H}_2}= 
\min_{\substack{\dim(\hat{\Sigma})\ = \ l \ \\
\hat{\Sigma} \text{ stable}}} \|\Sigma-\tilde{\Sigma}\|_{\mathcal{H}_2}.$$
Unfortunately, already in the linear case this is a highly nonconvex 
minimization problem such that finding a global minimizer is out of reach. 
Instead, we aim at constructing $\tilde{\Sigma}$ such that first-order 
necessary conditions for $\mathcal{H}_2$-optimality are fulfilled. In 
\cite{zhang2002}, the optimality conditions from 
\cite{Wil70} are extended to the bilinear case. More precisely,  it is 
shown that an 
$\mathcal{H}_2$-optimal reduced-order model is defined by a Petrov-Galerkin 
projection of the original model. Given a reduced-order system $\hat{\Sigma},$ 
let us consider the associated error system
\begin{equation}\label{error_system}
  \begin{aligned}
    A_e = \begin{bmatrix} A & 0 \\ 0 & \hat{A} \end{bmatrix}, \quad 
  \end{aligned}
  \begin{aligned}
    N_{k,e} = \begin{bmatrix} N_k & 0 \\ 0 & \hat{N}_k \end{bmatrix}, \quad 
  \end{aligned}
  \begin{aligned}
    B_e = \begin{bmatrix} B \\ \hat{B} \end{bmatrix}, \quad 
  \end{aligned}
  \begin{aligned}
    C_e = \begin{bmatrix} C & - \hat{C} \end{bmatrix},
  \end{aligned}
\end{equation}
as well as the generalized Lyapunov equations associated with it
\begin{equation}\label{eq:err_genlyap}
\begin{aligned}
  A_e P_e + P_e A_e^* + \sum_{k=1}^m N_{k,e} P_e N_{k,e}^* + B_e B_e^* &= 0, \\
  A_e^* Q_e + Q_e A_e + \sum_{k=1}^m N_{k,e}^* Q_e N_{k,e} + C_e^* C_e &= 0.
\end{aligned}
\end{equation}
Assuming the partitioning 
\begin{equation}\label{part_Gramians}
  P_e = \begin{bmatrix} P & X \\ X^* & \hat{P} \end{bmatrix}, \quad 
  Q_e = \begin{bmatrix} Q & Y \\ Y^* & \hat{Q} \end{bmatrix},
\end{equation}
the first-order necessary optimality conditions now are
\begin{equation}\label{h2_oc}
 \begin{aligned}
   Y^* A X + \hat{Q}^*\hat{A} \hat{P} &= 0, && Y^* N_k X + \hat{Q}^* 
\hat{N}_k \hat{P} &= 0, \\   Y^* B + \hat{Q}^*\hat{B} &= 0, && CX - 
\hat{C}\hat{P}&=0.
 \end{aligned}
\end{equation}
In \cite{zhang2002} the authors have proposed a gradient flow technique to 
construct a reduced-order model satisfying \eqref{h2_oc}. Since here we are 
interested in computations for large-scale systems for which this technique is 
not feasible, we instead use the iterative method from \cite{Benner2012}. The 
main idea is inspired by the \emph{iterative rational Krylov algorithm} from 
\cite{GugAB08} and relies 
on solving generalized Sylvester equations of the form
\begin{align*}
  AX + X\hat{A}^* + \sum_{k=1}^m N_k X \hat{N}_k^* + B \hat{B}^* &= 0, \\
  A^*Y + Y\hat{A} + \sum_{k=1}^m N_k^* Y \hat{N}_k - C^* \hat{C} &= 0.
\end{align*}
Based on a given reduced-order model 
$(\hat{A}_i,\hat{N}_{k,i},\hat{B}_i,\hat{C}_i)$, the subspaces spanned by 
columns of the 
solutions $X_i,Y_i \in \C^{n\times l}$ are used to generate an updated 
reduced-order model. More precisely, given unitary matrices $V_i,W_i \in 
\mathbb C^{n\times l}$ such that $\mathrm{span}(V_i)=\mathrm{span}(X_i)$ and 
$\mathrm{span}(W_i)=\mathrm{span}(Y_i),$ we set
\begin{align*}
  \hat{A}_{i+1} &= (W_i^*V_i)^{-1}W_i^*AV_i,  \ \ \hat{N}_{i+1} = 
(W_i^*V_i)^{-1}W_i^*N_kV_i, \\
 \hat{B}_{i+1} &= (W_i^*V_i)^{-1}W_i^*B , \ \ \hat{C}_{i+1}=CV_i.
\end{align*}
This type of fixed-point iteration is repeated until the 
reduced-order model is numerically converged up to a prescribed tolerance. For 
more details on the iteration, we also refer to 
\cite{Benner2012}.  

\section{Balanced model reduction for bilinear systems}\label{sec:balancing}

We shall briefly explain model reduction based on balancing controllability and 
observability. To this end we assume that the generalized Gramian matrices 
$P,Q$ are both Hermitian positive definite which is guaranteed by the 
assumption 
that 
the bilinear system (\ref{bilin1})--(\ref{bilin2}) is completely controllable 
and observable:\label{ass2} \\

\textbf{Assumption 4:} The matrix pair $(A,B)$ is controllable, i.e., 
\[
{\rm rank}(B\,AB\,A^{2}B\ldots A^{n-1}B)=n\,.
\]

\textbf{Assumption 5:} The matrix pair $(A,C)$ is observable, i.e., 
\[
{\rm rank}(C^{*}\,A^{*}C^{*}\,A^{2}C^{*}\ldots A^{n-1}C^{*})=n\,.
\]

\subsection{Singularly perturbed bilinear systems}
\label{sec:PT}
We consider a balancing transformation $x\mapsto T^{-1}x$ under which the 
Gramians transform according to \cite{moore1981}
\begin{equation}\label{balance1}
T^{-1}Q \left(T^{-1}\right)^{*}=\Sigma = 
T^*PT\,,
\end{equation}
where the diagonal matrix $\Sigma={\rm 
diag}(\sigma_{1},\sigma_2,\ldots,\sigma_{n})$ with $\sigma_1\ge 
\sigma_2\ge\ldots\ge\sigma_n>0$ contains the real-valued Hankel singular values 
(HSV) of the system. 
 Under the linear map $T$, the coefficients of (\ref{bilin1})--(\ref{bilin2}) 
transform according to 
\begin{equation}\label{balance2}
(A,N_{k},B,C)\mapsto (T^{-1}AT,\,T^{-1}N_{k}T,\,T^{-1}B,\, CT)\,,\quad 
k=1,\ldots,m\,.
\end{equation}
As the Hankel singular values are the square roots of the eigenvalues of the 
product $QP$, they are independent of the choice of coordinates. 
It can be shown (e.g.~\cite{antoulas2005}) that a balancing transformation that 
makes the two Gramians $Q$ and $P$ equal and diagonal is given by the matrix 
$T=\Sigma^{-\frac{1}{2}} V^T R$ with inverse $T^{-1}=S^TU\Sigma^{-1/2}$ where 
the matrices $U,V,S,R$ are defined by the Cholesky decompositions  $P=S^TS$ and 
$Q=R^TR$ of the two  Gramians solving \eqref{P} and \eqref{Q}, and their 
singular value decomposition $SR^T=U\Sigma V^T$.

Now suppose that $\Sigma=(\Sigma_1,\Sigma_2)$ with $\Sigma_{1}\in\R^{d\times 
d}$ and $\Sigma_{2}\in\R^{(n-d)\times(n-d)}$ corresponding to the splitting of 
the system states into relevant and irrelevant states. Further assume that 
$\Sigma_{2}\ll\Sigma_{1}$ in the sense that the smallest entry 
of $\Sigma_{1}$ is much larger than the largest entry of $\Sigma_{2}$. 
The rationale of balanced model reduction is based on a continuity argument: if 
the 
space of the uncontrollable and unobservable states is spanned by the singular 
vectors corresponding to $\Sigma_{2}=0$, then, by continuity of the solution of 
(\ref{bilin1})--(\ref{bilin2}) on the system's coefficients, small singular 
values should indicate hardly controllable and observable states that do not 
contribute much to the input-output behavior of the system.    

Using the notation $\Sigma_{2}=\cO(\eps)$ with $0<\eps\ll 1$ and partitioning 
the balanced coefficients according to the splitting into large and small HSV, 
then yields the following singularly perturbed system of equations (see 
\cite{hartmann2013,hartmann2010}): 
\begin{equation}\label{SF0}
\begin{split}
\frac{\dd z^{\eps}_{1}}{\dd t} & = \tilde{A}_{11}z^{\eps}_{1} + 
\frac{1}{\sqrt{\eps}} \tilde{A}_{12}z^{\eps}_{2} + 
\sum_{k=1}^{m}\left(\tilde{N}_{k,11}z^{\eps}_{1} + \frac{1}{\sqrt{\eps}} 
\tilde{N}_{k,12}z^{\eps}_{2} + \tilde{b}_{k,1} \right) u_{k}\\
\sqrt{\eps}\frac{\dd z^{\eps}_{2}}{\dd t} & = \tilde{A}_{21}z^{\eps}_{1} + 
\frac{1}{\sqrt{\eps}} \tilde{A}_{22}z^{\eps}_{2} +  
\sum_{k=1}^{m}\left(\tilde{N}_{k,21}z^{\eps}_{1} + \frac{1}{\sqrt{\eps}} 
\tilde{N}_{k,22}z^{\eps}_{2} + \tilde{b}_{k,2} \right) u_{k}\\
y^{\eps} & = \tilde{C}_{1}z^{\eps}_{1} + 
\frac{1}{\sqrt{\eps}}\tilde{C}_{2}z^{\eps}_{2}\,
\end{split}
\end{equation}
Here $z=T^{-1}x$, with $z=(z_{1},z_{2})\in\C^{d}\times\C^{n-d}$, denotes the 
balanced state vector where the splitting into $z_{1}$, $z_{2}$ is in 
accordance 
with the splitting of the HSV into $\Sigma_{1}$ and $\Sigma_{2}$. The splitting 
of the balanced coefficients 
\begin{equation}\label{balCoeff}
\tilde{A}=T^{-1}AT,\,\tilde{N}_{k}=T^{-1}N_{k}T,\,\tilde{b}_{k}=T^{-1}b_{k},\,
\tilde{C}=CT\,
\end{equation}
into $\tilde{A}_{11}$, $\tilde{A}_{12}$ etc.~can be understood accordingly.

\subsection{An averaging principle for bilinear systems}\label{sec:adiabatic}

In order to derive reduced-order models of (\ref{bilin1})--(\ref{bilin2}), we 
consider the limit $\eps\to 0$ in (\ref{SF0}). This amounts to the limit of 
vanishing small HSV $\Sigma_{2}$ in the original bilinear system. 

We suppose that Assumptions 1--5 hold for all $\eps>0$. As we will show in 
Appendix \ref{sec:stability}, the results in \cite{BDRR2016} can be modified to 
show that the  matrices $\tilde{A}_{11}$ and $ \tilde{A}_{22}$ are Hurwitz, and 
that their eigenvalues 
are bounded away from the imaginary axis. In this case, BIBO stability of the 
system together with the assumptions on the admissible controls imply that 
$z^{\eps}_{2}\to 0$ pointwise for all $t> 0$ as $\eps\to 0$. However, the rate 
at which $z^{\eps}_{2}$ tends to zero and hence the limiting bilinear systems 
clearly depends on the controls $u$, especially when $u$ depends on $\eps$. We 
give only a formal justification of the different candidate equations that can 
be obtained in the limit of vanishing small HSV and refer to \cite{hartmann2013} 
for further details. 

\subsection{Balanced truncation}
\label{sec:bt}

If $z^{\eps}_{2}=o(\sqrt{\eps})$, we expect that the first two equations in 
(\ref{SF0}) decouple as $\eps\to 0$, which implies that the limiting bilinear 
system will be of the form 
\begin{equation}
\label{eq:bt}
\begin{split}
\frac{\dd z_{1}}{\dd t} & = \tilde{A}_{11}z_{1} + 
\sum_{k=1}^{m}\left(\tilde{N}_{k,11}z_{1} + \tilde{b}_{k,1} \right) u_{k}\\
y & = \tilde{C}_{1}z_{1} \,.
\end{split}
\end{equation}
The assumption that $z^{\eps}_{2}$ goes to zero faster than $\sqrt{\eps}$ is 
the basis of the traditional \emph{balanced truncation} approach in which the 
weakly controllable and observable degrees of freedom are eliminated by 
projecting the equations to the linear subspace 
\[
S_{1} = \{(z_{1},z_{2})\in\C^{n}\colon z_{2}=0\}\simeq \C^{d}.
\]  
The validity of the approximation for all $t\ge 0$ requires that 
$z_{2}^{\eps}(0)=0$; cf.~Remark \ref{rem:avg} below. 

\subsection{Singular perturbation} 
\label{sec:avg}

If $z^{\eps}_{2}=\cO(\sqrt{\eps})$ the $z_{1}$, $z_{2}$ equations do not 
decouple as $\eps\to 0$, and the limiting equation turns out to be different 
from (\ref{eq:bt}). To reveal it, it is convenient to introduce scaled 
variables 
by $z_{2}= \sqrt{\eps}\zeta$ by which (\ref{SF0}) becomes 
\begin{equation}\label{SF}
\begin{split}
\frac{\dd z^{\eps}_{1}}{\dd t} & = \tilde{A}_{11}z^{\eps}_{1} + 
\tilde{A}_{12}\zeta^{\eps} + \sum_{k=1}^{m}\left(\tilde{N}_{k,11}z^{\eps}_{1} + 
\tilde{N}_{k,12}\zeta^{\eps} + \tilde{b}_{k,1} \right) u_{k}\\
\eps\frac{\dd\zeta^{\eps}}{\dd t} & = \tilde{A}_{21}z^{\eps}_{1} + 
\tilde{A}_{22}\zeta^{\eps} +  \sum_{k=1}^{m}\left(\tilde{N}_{k,21}z^{\eps}_{1} 
+ 
\tilde{N}_{k,22}\zeta^{\eps} + \tilde{b}_{k,2} \right) u_{k}\\
y^{\eps} & = \tilde{C}_{1}z^{\eps}_{1} + \tilde{C}_{2}\zeta^{\eps}\,.
\end{split}
\end{equation}
Equation (\ref{SF}) is an instance of a slow-fast system with $z_{1}$ being the 
slow variable and $\zeta=z_{2}/\sqrt{\eps}$ being fast, and 
for non-pathological controls $u$, the averaging principle applies 
\cite{grammel1997}. The idea of the averaging principle is to average the fast 
variables in the equation for $z_{1}$ against their invariant measure, because 
whenever $\eps$ is sufficiently small, the fast variables relax to their 
invariant measure while the slow variables are effectively frozen, and 
therefore 
the slow dynamics move under the average influence of the fast variables. This 
clearly requires that the convergence of the fast dynamics is sufficiently fast 
and independent of the initial conditions. The auxiliary fast subsystem for 
frozen slow variable $z_{1}$ reads
\begin{equation}\label{fast1}
\frac{\dd\tilde{\zeta}}{\dd\tau} = \tilde{A}_{22}\left(\tilde{\zeta}+ 
\tilde{A}_{22}^{-1} \tilde{A}_{21}z_{1}\right) +  
\sum_{k=1}^{m}\left(\tilde{N}_{k,22}\left(\tilde{\zeta}+ \tilde{A}_{22}^{-1} 
\tilde{A}_{21}z_{1}\right) + \tilde{B}_{k,2} \right) \tilde{u}_{k}\,,
\end{equation}
with 
\begin{equation}\label{fast2}
\tilde{B}_{k,2} = \left(\tilde{N}_{k,21} - \tilde{N}_{k,22}\tilde{A}_{22}^{-1} 
\tilde{A}_{21}z_{1}\right) + \tilde{b}_{k,2}\,.
\end{equation}
It is obtained from (\ref{SF}) by rescaling the equations according to $\tau= 
t/ 
\eps$ and  $\tilde{\zeta}(\tau)=\zeta^{\eps}(\eps\tau)$, 
$\tilde{u}(\tau)=u(\eps\tau)$ and sending $\eps\to 0$. Since the admissible 
controls decay on time scales that are of order one in $t$ (i.e.~$\cO(1/\eps)$ 
in $\tau$), it follows that 
\begin{equation*}
\lim_{\tau\to\infty} \tilde{\zeta}(\tau;z_{1}) = - \tilde{A}_{22}^{-1} 
\tilde{A}_{21}z_{1}\,.
\end{equation*}
In other words, for fixed $z_{1}$ the fast dynamics converge to the Dirac mass 
$\delta_{m}$ at $m=-\tilde{A}_{22}^{-1} \tilde{A}_{21}z_{1}$. This can be 
rephrased by saying that for all admissible controls and in the limit $\eps\to 
0$ the dynamics (\ref{SF}) collapse to the invariant subspace 
\[
S_{2} = \{(z_{1},z_{2})\in\C^{n}\colon z_{2} = - \tilde{A}_{22}^{-1} 
\tilde{A}_{21}z_{1}\}\simeq \C^{d}.
\]
Averaging the fast variables in (\ref{SF}) against their invariant measure 
$\delta_{m}$, then yields the averaged equation for the slow variables: 
\begin{eqnarray}
\label{eq:avg}
\begin{aligned}
\frac{\dd z_{1}}{\dd t} & = \hat{A} z_{1} +  
\sum_{k=1}^{m}\left(\hat{N}_{k}z_{1} + \tilde{b}_{1,k} \right) u_k\\
y & = \hat{C}z_{1}\,, 
\end{aligned}
\end{eqnarray}
with the coefficients 
\begin{equation}
\label{eq:avgCoeff}
\begin{aligned}
\hat{A} & = \tilde{A}_{11} - \tilde{A}_{12}\tilde{A}_{22}^{-1}\tilde{A}_{21}\\
\hat{N}_{k} & = \tilde{N}_{k,11} - 
\tilde{N}_{k,12}\tilde{A}_{22}^{-1}\tilde{A}_{21}\\
\hat{C} & = \tilde{C}_{1} - \tilde{C}_{2}\tilde{A}_{22}^{-1}\tilde{A}_{21}\,.
\end{aligned}
\end{equation}

The situation here is special, in that the controls decay sufficiently fast so 
that the invariant measure of the fast variables is independent of $u$. For 
other choices of admissible controls, however, the invariant measure may depend 
on $u$, which then gives rise to averaged equations with measure-valued right 
hand side \cite{gaitsgory1992,grammel1997,vigodner1997}. The following 
approximation result has been proved in \cite{hartmann2013}; 
cf.~\cite{watbled2005}.

\begin{thm}\label{thm:avg}
Let $u=u^{\eps,\gamma}$ in (\ref{SF}) be admissible, satisfying $u(t) = 
u(t/\eps^{\gamma})$ for some $0<\gamma<1$. Further let $y^{\eps}(t)$  be the 
observed solution of (\ref{SF}) with consistent initial conditions 
$(z_{1}^{\eps}(0),\zeta^{\eps}(0))=(\eta,-\tilde{A}_{22}^{-1} 
\tilde{A}_{21}\eta)$, and let $\bar{y}(t)$ denote the output of the averaged 
equation (\ref{eq:avg}) on the bounded time interval $[0,T]$, starting from the 
same $z_{1}(0)=\eta$. Then there exists a constant $C=C(T)$, such that 
\begin{equation*} 
\sup_{0\le t\le T}\left|  y^{\eps}(t) - \hat{y}(t) \right| \le C 
\eps^{\gamma}\,.
\end{equation*}
\end{thm}

We should stress that it is possible to relax the condition on the initial 
conditions that guarantees that $(z_{1}^{\eps}(0),\zeta^{\eps}(0))\in S_{2}$. 
In 
this case there will be a transient initial layer of thickness 
$\cO(\sqrt{\eps})$, in which there is a rapid adjustment of the initial 
conditions to the invariant subspace $S_{2}$ and during which the averaged 
dynamics deviates from the original dynamics, with an $\cO(1)$ error. A uniform 
approximation on $[0,T]$ can then be obtained by a so called \emph{matched 
asymptotic expansion} that matches an initial layer approximation with the 
averaged dynamics \cite{omalley1991}.

\begin{rem}\label{rem:avg}
 For single-input systems ($m=1$), a sufficient condition for BIBO stability of 
(\ref{bilin1}) is that $A$ is Hurwitz, in which case there exists a $\delta>0$, 
such that $A+s N$ is Hurwitz for all $s\in[-\delta,\delta]$. 
%
The stability of $A$ is 
inherited by the Schur complement $\hat{A}$, in (\ref{eq:avg}) and consequently 
$\hat{A} + s\hat{N}$ inherits stability, with a possibly smaller stability 
region. (See Appendix \ref{sec:stability} for details.) Hence reduced 
single-input systems are again BIBO stable.
\end{rem}


\section{Numerical details}
\label{sec:numeric}

Before  testing $\mathcal{H}_{2}$ and balanced model reduction for examples from 
stochastic control and quantum dynamics, see Secs.~\ref{sec:fpe} and 
\ref{sec:lvne}, respectively, we will first focus 
on the numerical issues related to the scaling of the controls and the 
preprocessing of the unstable $A$ matrix. 

\subsection{Structured bilinear systems}
\label{sec:stationary}
The subsequent numerical examples share several special properties that result  
from a physical interpretation and that require a careful numerical 
treatment. In this section, we provide some insight in how the model reduction 
methods are applied to the particularly structured bilinear systems. In fact, 
in the FPE as well as in the LvNE context, the initial setup leads 
to a purely bilinear system of the form
\begin{equation}\label{eq:pur_bil}
  \begin{aligned}
 \dot{x}(t) &= Ax(t) + \sum_{k=1}^m N_kx(t)u_k(t), \quad x(0)=x_0, \\
  y(t)&=Cx(t).
  \end{aligned}
\end{equation}
In either case, the system exhibits a nontrivial stationary solution $x_e$ 
corresponding to a simple eigenvalue $0$ of the system matrix $A,$ i.e., 
$Ax_e=0.$ For the applications we are interested in the deviation of the state 
$x$ from the stationary solution. Let us therefore introduce the reference 
state $\tilde{x}=x-x_e$ that is governed by the bilinear system
\begin{equation}\label{eq:pur_bil_shifted}
 \begin{aligned}
   \dot{\tilde{x}}(t)&= A\tilde{x}(t) + \sum_{k=1}^m N_k \tilde{x}(t) u_k(t) + 
\underbrace{\begin{bmatrix}N_1 x_e,\dots,N_kx_e\end{bmatrix}}_{B}u(t), \quad 
\tilde{x}(0) = x_0-x_e, \\ 
y(t) &= C\tilde{x}(t)+Cx_e,
 \end{aligned}
\end{equation}
where the term $Cx_e$ can be interpreted as a constant nonzero feedthrough $D$ 
of the system. For the reduced-order model, we thus may simply set 
$\hat{D}=Cx_e$ 
such that we can simply focus on the output operator $C.$ In accordance with 
standard model reduction concepts that assume a homogeneous initial condition, 
here we assume that the initial state of the original system is the 
equilibrium, i.e., $\tilde{x}(0)=x_e-x_e=0.$ While the system now has been 
transformed from a purely bilinear into a standard bilinear system, we still 
have to deal with the problem of a system matrix that is not asymptotically 
stable. In what follows, we present two different techniques that bypass this 
problem. 

\subsection{Sparsity preserving projection}
\label{sec:project}
In our examples, the system matrices are mass and positivity preserving. 
Numerically this is reflected in the fact that the system matrix $A$ as well as 
the bilinear coupling matrices have zero row sum. In other words, the 
vector $\mathbf 1_n:=\begin{bmatrix}1,\dots,1\end{bmatrix}^*\in \C^n$ 
satisfies $\mathbf 
1_n^* A = \mathbf 1_n^*N_k = 0.$ The intuitive idea now is splitting the state 
into the direct sum of the asymptotically stable subspace and the 
eigenspace associated with the eigenvalue $0.$ Since a straightforward 
implementation in general will destroy the sparsity pattern of the matrices, we 
suggest to use a particular decomposition that has been introduced in a 
 similar setup in \cite{BreKP16}. Define the matrix
\begin{align*}
  R = \begin{bmatrix} I & 0 \\ 0 & 0 \end{bmatrix}+ x_e e_n^* - e_n 
\begin{bmatrix}\mathbf{1}_{n-1}^* & 0 \end{bmatrix},
\end{align*}
where $e_n$ denotes the $n$-th unit vector in $\C^n.$ An easy calculation now 
shows that the inverse $R^{-1}$ is given as
\begin{align*}
  R^{-1} =  \begin{bmatrix} I & 0 \\ 0 & 0 \end{bmatrix} + e_n \mathbf{1}^* - 
\begin{bmatrix} \tilde{x}_e \\ 0 \end{bmatrix} \mathbf{1}^*, 
\end{align*}
where the vector $\tilde{x}_e \in \C^{n-1}$ consists of the first $n-1$ 
components of $x_e \in \C^n.$ Assume that the matrices $A,N_k$ and $B$ are 
partitioned as follows
\begin{align*}
  A = \begin{bmatrix} \tilde{A} & A_{(1:n-1,n)} \\ * & * \end{bmatrix}, \ \ 
  N_k = \begin{bmatrix} \tilde{N}_k & N_{k,(1:n-1,n)} \\ * & * \end{bmatrix}, \ 
\ B=\begin{bmatrix} \tilde{B} \\ * \end{bmatrix},
\end{align*}
with $\tilde{A},\tilde{N}_k \in \C^{n-1\times n-1 }$ and $\tilde{B} \in 
\C^{n-1\times m}.$ Finally, a state space transformation $z:=R^{-1}\tilde{x}$ 
yields the equivalent bilinear system
\begin{equation}\label{eq:bil_transformed}
  \begin{aligned}
  \dot{z}(t) &= (R^{-1}AR) z(t) + \sum_{k=1}^m (R^{-1}N_k R) z(t) u_k(t) + 
(R^{-1}B) u(t), \quad z(0)= 0, \\
y(t) &= (CR) z(t) + Cx_e.
\end{aligned}
\end{equation}
Making use of the relations $Ax_e=0=A^* \mathbf{1}_n=N_k^*\mathbf{1}_n,$ we 
conclude that the last row of $R^{-1}AR,R^{-1}N_kR$ and $R^{-1}B=R^{-1}N_k x_e$ 
is zero. This implies that the last component of $z(t)$ is constant, and, due 
to $z(0)=0$ vanishes for all times $t.$ As a consequence, we can focus on the 
first $n-1$ components $\tilde{z}(t)$ of $z(t)$ which, after some calculations, 
can be shown to satisfy
\begin{equation*}
\begin{aligned}
  \tilde{z}(t) &=( \tilde{A} - A_{(1:n-1,n)} \mathbf{1}_{n-1}^* ) \tilde{z}(t) 
+ \sum_{k=1}^m ( \tilde{N}_k - N_{k,(1:n-1,n)} \mathbf{1}_{n-1}^* ) 
\tilde{z}(t)u_k(t) + \tilde{B}u(t), \\
y(t) &= \tilde{C}\tilde{z}(t) + Cx_e, \quad \tilde{z}(0) = 0.
\end{aligned}
\end{equation*}
Typically, the matrices $A$ and $N_k$ result from finite difference or finite 
element discretization, respectively, and thus are sparse. The previous 
projection in fact only slightly increases the number of nonzero entries. 
Moreover, the matrices are given as the sum of the original data and a low rank 
update which can be exploited in a numerical implementation as well.

\subsection{Discounting the system state}
\label{sec:discount}
An {\em ad-hoc} alternative to the decomposition of the state space into stable 
and unstable directions is the ``shifting'' of the $A$-matrix by a translation 
$A\mapsto A-\alpha I$ for some  $\alpha>0$. If $A$ has a simple eigenvalue zero, 
as in our case, there exists an $\alpha>0$, such that the matrix $A-\alpha I$ is 
Hurwitz. 
For linear systems the shifting can be interpreted as a discounting of the 
controllability and observability functionals that renders the associated 
Gramians finite \cite{Nichols2011}. 

As the controllability and observability Gramians in the bilinear case are 
lacking a similar interpretation, the shifting has no clear functional analogue 
(cf.~\cite{Benner2011a}). 
It is still possible to stabilize the system by a joint state-observable 
transformation
\[
(x,y)\mapsto (e^{-\alpha t}x,e^{-\alpha t}y)=:(\tilde{x},\tilde{y})\]
under which the system (\ref{bilin1})--(\ref{bilin2}) transforms according to  
\begin{equation}\label{bilinShifted}
\begin{aligned}
  \frac{\dd \tilde{x}}{\dd t} & = \left(A - \alpha I\right) \tilde{x} + 
\sum_{k=1}^{m} \left(N_k \tilde{x}  + 
b_{k}\right)u_{k}\,,\quad \tilde{x}(0)=x_{0}\\
\tilde{y} & = C\tilde{x}\,.
\end{aligned}
\end{equation}
Even though (\ref{bilinShifted}) and (\ref{bilin1})--(\ref{bilin2}) are 
equivalent as state space systems, the shifting clearly affects the Hankel 
singular value spectrum and, as a consequence, the reduced system. (As a matter 
of fact, the Hankel singular values do not even exist in case of the 
untransformed system.) Hence the parameter $\alpha$ should be regarded as a 
regularization parameters that must chosen as small as possible. 

Later on we compare stabilization of the $A$ matrix by state space decomposition 
and shifting in terms of the achievable state space reduction (i.e., decay of 
Hankel singular values) and fidelity of the reduced models.

\subsection{Scaling the control fields}
\label{sec:scaling}

Assumption 1 in Sec.~\ref{sec:assumptions} deals with the existence and 
uniqueness of controllability and observability Gramians which are obtained as 
solutions to the generalized Lyapunov equations.
The criterion given there involves an upper bound for the matrix 2-norm of the 
control matrices $N_k$.
In the examples of model order reduction shown below, this can be achieved by a 
 
suitable scaling $u \mapsto \eta u, N_k\mapsto N_k/\eta, B \mapsto B/\eta$ with 
real $\eta>1$ which leaves the equations of motion invariant but, clearly, not 
the Gramians. Hence, by increasing $\eta$, we drive the system to its linear 
counterpart. For the limit $\eta\rightarrow\infty$, the system matrices $N$ and 
$B$ vanish and we obtain a linear system. For this reason, $\eta$ should not be 
chosen too large. 

\subsection{Calculation of the $\mathcal{H}_2$ error}
\label{sec:h2error} 
To quantify the error introduced by dimension reduction, we use the 
$\mathcal{H}_2$-norm introduced in Sec.~\ref{sec:h2}. We emphasize that the 
effort required 
for computing the $\mathcal{H}_2$-error is negligible when compared to solving 
the generalized Lyapunov equations arising for balanced truncation and singular 
perturbation, respectively, which is seen as follows. Given a reduced-order 
system $\hat{\Sigma},$ the associated $\mathcal{H}_2$-error is given as
\begin{equation}
\label{eq:H2error}
 \|\Sigma -\hat{\Sigma}\|_{\mathcal{H}_2}^2 = \mathrm{tr}(C_e P_e C_e^*),
\end{equation}
where $P_e$ solves \eqref{eq:err_genlyap}. Using the particular structure of 
the error system, this is obviously the same as
\begin{equation*}
 \|\Sigma -\hat{\Sigma}\|_{\mathcal{H}_2}^2 = \mathrm{tr}(CPC^*) - 
2\mathrm{tr}(CX\hat{C}^*) + \mathrm{tr}(\hat{C}\hat{P}\hat{C}^*).
\end{equation*}
However, the term $\mathrm{tr}(CPC^*)$ now can be precomputed since $P$ is 
required for the balancing-based methods anyway. What remains is the 
computation of the solutions $X$ and $\hat{P}$ of the following the generalized 
Sylvester and Lyapunov equations, respectively
\begin{align*}
  AX + X\hat{A}^* + \sum_{k=1}^m N_k X \hat{N}_k^* + B\hat{B}^* &=0,\\
  \hat{A}\hat{P} + \hat{P}\hat{A}^* + \sum_{k=1}^m \hat{N}_k \hat{P} 
\hat{N}_k^* + \hat{B}\hat{B}^* &=0.
\end{align*}
Based on the results from \cite{Dam08}, we can compute $X=\lim_{i\to 
\infty}X_i$ and $\hat{P}=\lim_{i\to \infty} \hat{P}_i$ as the limits of 
solutions to standard Sylvester and Lyapunov equations 
\begin{align*}
AX_1 + X_1 \hat{A}^* + B\hat{B}^* &=0, \\
  AX_i + X_i\hat{A}^* + \sum_{k=1}^m N_k X_{i-1} \hat{N}_k^* + B \hat{B}^* &=0, 
\quad i\ge 2, \\
\hat{A}\hat{P}_1 + \hat{P}_1 \hat{A}^* + \hat{B}\hat{B}^* &=0, \\
  \hat{A}\hat{P}_i + \hat{P}_i\hat{A}^* + \sum_{k=1}^m \hat{N}_k \hat{P}_{i-1} 
\hat{N}_k^* + \hat{B} \hat{B}^* &=0, 
\quad i\ge 2.
\end{align*}

\subsection{Software}

All of the numerical tests of the dynamical systems presented in the following 
have been carried out using the  \textsc{WavePacket} software project which 
encompasses all numerical methods for model order reduction as discussed above.
Being hosted at the open--source platform Sourceforge.net, this program package 
is publicly available, along with many instructions and demonstration examples, 
see \texttt{http://sf.net/projects/wavepacket} and 
Refs.~\cite{BSchmidt:75,BSchmidt:78}.
In addition to a mature \matlab version, there is also a C++ version currently 
under development.

\section{Fokker--Planck equation}
\label{sec:fpe}

We start off with an example from stochastic control in classical mechanics: a 
semi-discretized Fokker--Planck equation (FPE) with external forcing. To this 
end, we consider the stochastic differential equation
\begin{equation}\label{SDE}
\dd X_t = \left(u_t - \nabla V(X_t)\right)\dd t + \sigma \dd W_t\,, \quad 
X_0=x\,,
\end{equation}
that governs the motion of a classical particle with position $X_t\in\R^n$ at 
time $t>0$. The motion is influenced by the gradient of a smooth potential $V$, 
a deterministic control force $u$ and a random forcing coming from the 
increments of the Brownian motion $(W_t)_{t\ge 0}$ in $\R^n$. For simplicity we 
assume that the potential $V$ is $C^\infty$, with 
\[
V(x)\sim |x|^{2k}\quad\textrm{as}\quad |x|\to\infty.
\]
Note that $X_t=X_t(\omega)$ is a random variable for every $t>0$, and an 
equivalent characterization of the diffusion process $X_t$ is in terms of its 
probability distribution
\[
\int_A \rho(y,t)\,\dd y = \textrm{Prob}[X_t \in A \,|\, X_0=x]
\]
where $A\subset\R^n$ is any measurable (Borel) subset of $\R^n$, and  
$\rho\colon\R^n\times\R_+\to\R_+$ is the associated probability density whose 
time evolution is governed by the Fokker--Planck equation
\begin{equation}\label{fokker}
\frac{\partial \rho}{\partial t} = \nabla\cdot\left(\beta^{-1}\nabla \rho + 
\rho(\nabla V - u)\right)\,,\quad  \lim_{t\searrow 0}\rho(\cdot,t) = \delta_x\,,
\end{equation} 
with the shorthand $\beta=2/\sigma^2$ for the inverse temperature. 
The limit in the last equation, that must be understood in the sense of weak 
convergence of probability measures (or, equivalently, weak-$*$ convergence), 
reflects our choice of deterministic initial condition $X_0=x$; the 
regularization property of the parabolic FPE guarantees that $\rho(\cdot,t)$ is 
$C^2$ for any $t>0$; moreover the solution stays non-negative. 
Later on, we will consider the case that the initial conditions are drawn from 
a 
probability density $\rho_0$ and thus replace $\delta_x$ by $\rho_0$. 

Note that by the divergence theorem, 
\begin{equation}\label{L1cons}
\frac{\dd}{\dd t}\int \rho(y,t)\,\dd y = 0\,,
\end{equation}
hence the total probability is conserved along the solution of (\ref{fokker}).

When $u=u_0$ is constant, the properties of the potential $V$ entail that the 
solution to the FPE converges exponentially fast to a stationary solution 
$\rho_\infty$ as $t\to\infty$ (see, e.g., \cite{lelievre2016}). The stationary 
solution is then given as the unique normalized solution to the elliptic 
partial 
differential equation
\begin{equation}\label{canFPE}
0= \nabla\cdot\left(\beta^{-1}\nabla \rho + \rho\nabla V_u \right)
\end{equation}
and has the form 
\begin{equation}\label{canDens}
\mu(x) = \frac{1}{Z_u}e^{-\beta V_u(x)}\,,\quad Z_u = \int_{\R^n}e^{-\beta 
V_u(x)}\,\dd x\,,
\end{equation}
where we have introduced the shorthand $V_u(x)=V(x)-u_0\cdot x$ for the tilted 
potential.

Later on we will study the convergence towards the stationary distribution that 
is exponential with a rate essentially given by the first non-zero eigenvalue 
$-\lambda_1>0$, and compare the fully discretized model with its reduced-order 
approximant.

\subsection{Metastable model system}
\label{sec:fpe_model}

We consider the situation of a diffusive particle in $\R^2$ that is confined by 
the following periodically perturbed quadruple-well potential\footnote{Eric 
Barth, private communication.} shown in 
Figure~\ref{fig:fpe_potential}
\begin{equation}\label{poten}
\begin{aligned}
V =\; & 0.01\left( (x_1-0.1)^4 - 20x_1^2  + (x_2+0.4)^4 - 20 x_2^2\right.\\  
& + \left.10\sin(5x_1)\cos(5x_2) + x_1 x_2 + 290.4 \right)
  \end{aligned}
 \end{equation}
The potential has a deep energy well in the south-east of the 
$x_1$-$x_2$-plane, one slightly shallower well in the south-west and two even 
shallower wells in the north-west and north-east.
The system is metastable, in that the time scale to reach the deepest potential 
energy well from any of the other three wells is of the order of the Arrhenius 
timescale $e^{\beta\Delta V_{\rm min}}\gg 1$ where $\Delta V_{\rm min}$ denotes 
the minimum energy barrier that a particle going from one well to the 
south-east 
well would have to overcome \cite{Berglund2013}. The various local minima of 
the 
 potential energy surface that originate from the periodic perturbation do not 
have any significant effect on the transition rates between the main wells. 
The corresponding stationary density $\mu$ is shown in the upper left panel of 
Fig.~\ref{fig:fpe_cluster1}. For moderate temperature ($\beta=4.0$) essentially 
only the two main wells are populated, with considerably more weight on the 
deepest minimum (SE).

\subsection{Finite difference discretization} 
\label{sec:fpe_discrete}

Sine all coefficients in the FPE (\ref{fokker}) are sufficiently smooth, we can 
discretize it using finite differences. Let $\Omega = (a, b)\times (c, d)$, and 
consider the solution domain $D=\bar{\Omega}\times [0,T]\subset\R^2\times\R_+$. 
On a bounded domain, probability conservation (\ref{L1cons}) requires that the 
outwards probability flux 
\begin{equation*}
J_u(\rho) = \beta^{-1}\nabla\rho + \rho(\nabla V - u)
\end{equation*}
across the boundary of the spatial domain is zero at any time. Letting $\nu$ 
denote the outward pointing normal to $\partial\Omega$, the FPE (\ref{fokker}) 
on $D$ reads  
\begin{equation} \label{fokker2}
\begin{aligned}
\frac{\partial \rho}{\partial t} = \nabla\cdot\left(\beta^{-1}\nabla \rho + 
\rho(\nabla V - u)\right)\,, & \quad  (x,t) \in \Omega \times (0,T]\\
0  = \nu\cdot J_u(\rho)\,, & \quad  (x,t)\in \partial\Omega \times [0,T]\\
\rho_{0} = \rho\,, & \quad (x,t)\in\Omega\times \{0\}\,.
\end{aligned}
\end{equation}
For simplicity we will discretize the equation on the uniform mesh 
\[
\begin{aligned}
  \Omega_h & := \{(a+ih_1,c+jh_2)\;\colon 1<i<n_1-1, 1<j<n_2-1\},\\
  \partial\Omega_h & := \{(a+ih_1,c+jh_2)\;\colon 0\leq i \leq n_1, 0\leq j\leq 
n_2\}\setminus  \Omega_h,
\end{aligned}
\]
where $h_1=(b-a)/(n_1+1)$ and $h_2=(d-c)/(n_2+1)$ are the mesh sizes in $x_1$ 
and $x_2$ direction. Letting $w_{i,j}=\rho(x_{1,i},x_{2,j})$ with 
$(x_{1,i},x_{2,j})\in\Omega$, we approximate the first and second derivatives 
in 
the usual way by centered finite differences, e.g.
\begin{equation}
\begin{aligned}
\left.\frac{\partial\rho}{\partial x_1}\right|_{x=(x_{1,i},x_{2,j})} & \approx  
\frac{w_{i+1,j}-w_{i-1,j}}{2h_1}\\
\left.\frac{\partial^2\rho}{\partial x_1^2}\right|_{x=(x_{1,i},x_{2,j})} 
& \approx \frac{w_{i+1,j}-2w_{i,j} + w_{i-1,j}}{h_1^2}\,.
\end{aligned}
\end{equation}

For sufficiently small mesh size $h=(h_1,h_2)$, the finite difference 
discretization is known to preserve positivity, norm and stochastic stability. 
As a consequence, the stationary distribution of the discretized equation is 
the 
unique asymptotically stable fixed point and approximately equal to the 
stationary solution $\mu$ of the original equation, evaluated at the grid 
points; cf.~\cite{latorre2011}. 

In matrix-vector notation, the discretization of (\ref{fokker2}) can be 
compactly written as
\begin{equation}\label{discFPE}
	\dot{v} = Av + \sum_{k=1}^{2}u_k N_k v \,,\quad v(0)=v_{0}\,
\end{equation}
where $v\in\R^{n}$ with $n=n_1 n_2$ is the column-wise tensorization of 
$(w_{i,j})_{i,j}$, i.e. $v_{i + (j-1)n_1} = w_{i,j}$, $A\in\R^{n\times n}$ is 
the discretization of the Fokker--Planck operator 
\[
 \nabla\cdot\left(\beta^{-1}\nabla \rho + \rho\nabla V \right) = 
\beta^{-1}\Delta\rho + \nabla V \cdot\nabla\rho + (\Delta V)\rho
\]
of the uncontrolled dynamics, and the $N_i$ are the discretization of the 
partial derivatives $\partial/\partial x_i$ on the tensorized grid, $u_1$ and 
$u_2$ are the components of $u$. 

By construction, $-A$ is an $M$-matrix with a simple eigenvalue $0$ that 
corresponds to the discretized unique stationary distribution 
$\pi\approx\mu|_{\Omega_h}$, all other eigenvalues have strictly negative real 
parts. This is in contrast to the spectral properties of the original operator 
that is symmetric (essentially self-adjoint) when considered on the 
appropriately weighted Hilbert space, i.e., all its eigenvalues are real. We 
observe, however, that the dominant eigenvalues are real when the 
discretization 
is sufficiently fine. 

Tab.~\ref{tab:fpe_spectrumA} gives the 12 smallest eigenvalues (by their 
magnitude) of the matrix $A$ and  for a discretization of the domain 
$\Omega=(-6.0, 6.0)\times (-5.5, 6.5)$ with uniform mesh size $h_1=h_2=0.25$; 
the size of the resulting matrix $A$ is $2401\times 2401$. 
The $L^1$-deviation between the eigenvector $\pi$ to the eigenvalue 
$\lambda_0=0$ and $\mu$ evaluated at the grid points is smaller than $0.007$. 
As 
the theory predicts, the matrix has 4 dominant eigenvalues close to $0$ 
(including $\lambda_0=0$) that are separated from the rest of the spectrum. 
Figure~\ref{fig:fpe_cluster1} shows the 4 dominant eigenvectors of $A$, the 
first one being the stationary distribution that is essentially supported by 
the 
two deepest minima, the second one describing the dominant transition process 
between the deepest and the second deepest minimum, the third one representing 
the transitions between the second and the third deepest minimum and so on. 
The absolute values of the corresponding eigenvalues 
$\lambda_1,\,\lambda_2,\,\lambda_3<0$ represent (up to an error of order 
$\sqrt{\eps}$) the transition rates between the dominant potential energy 
wells. 
The fact that the subdominant eigenvalues appear in clusters of 4  has to do 
with the approximate four-fold symmetry of the potential. 
By tilting the potential towards one or several of the minima (thus flattening 
some of the other minima) the number of eigenvalues in the dominant cluster 
changes according to the number of resulting wells.

\subsection{Stable input-output system in standard form}

We first augment (\ref{discFPE}) by an output equation. To this end we 
introduce 
the observable $y=(y_1,\ldots,y_4) \geq 0$ denoting the probability for each of 
the four energy wells. The $y_i$ are given by summation of the density $v$ over 
all mesh points corresponding to the four quadrants of the $x_1$-$x_2$-plane, 
which, using the tensorized form of the equation, can be written as 
\begin{equation}\label{obsFPE}
y = Cx
\end{equation} 
for a matrix $C\in \R^{4\times n}$. The discretized FPE is bilinear, but it is 
homogeneous, i.e., it does not contain a purely linear term ``$Bu$'', which 
implies that no state is reachable from the origin $v(0)=0$.\footnote{Note that 
$v(0)=0$ is not a probability density, hence not an admissible starting point 
from a probabilistic point of view.} To transform (\ref{discFPE}) into the 
standard form (\ref{bilin1}), we follow the procedure described in Sec. 
\ref{sec:project}.

\subsection{Numerical results}
\label{sec:fpe_results}

Here and throughout the following we will use the following short-hand notation 
when comparing results for the three approaches to model order reduction:
BT stands for balanced truncation, as given by equation (\ref{eq:bt}) in 
Sec.~\ref{sec:bt} whereas SP symbolizes the averaging principle derived from 
singular perturbation theory, as given by equations 
(\ref{eq:avg})--(\ref{eq:avgCoeff}) in Sec.~\ref{sec:avg}. 
Finally, H2 is the $\mathcal{H}_2$-optimal model order reduction of 
Sec.~\ref{sec:h2}.

The details of the following comparisons depend sensitively on the value of the 
parameter $\eta$ used for scaling of the control field $u(t)$ and matrices 
$N_k$ 
and $B$, which is necessary to guarantee existence and uniqueness of 
controllability and observability Gramians, see Sec.~\ref{sec:scaling}.
For the particular example of the FPE dynamics for inverse temperature $\beta=4$ 
investigated here, we use a value of $\eta=10$ consistently for all three 
approaches to model order reduction.
Moreover, to stabilize the $A$ matrix we use here the projection method from 
Sec.~\ref{sec:project}. 
However, our results are practically unchanged when using the discounting 
approach described in Sec.~\ref{sec:discount} instead, assuming that the 
regularization parameter $\alpha$ is within a reasonable range.

The behavior of the $\mathcal{H}_2$-error defined in (\ref{eq:H2error}) for the 
discretized FPE is shown in Fig.~\ref{fig:fpe_h2error}. 
Similarly for all of the three methods, this error displays a plateau value of 
approximately $10^{-5}$ for a reduced dimensionality of about $d \gtrsim 60$. 
Upon further reduction of the dimensionality we observe a rapid increase over 
several orders of magnitude indicating a decreased quality when reducing overly.
In most cases it is found that the $\mathcal{H}_2$-error for the H2 method is 
slightly lower than for BT, which in turn is slightly lower than for the SP 
method.

While the $\mathcal{H}_2$-error characterizes the error of model order 
reduction 
for the limiting case of an infinitely short pulse (Dirac-like) control field, 
it may be also of interest to compare full versus reduced order models for more 
realistically shaped control fields.
As an example we consider here the Fokker--Planck dynamics, again for $\beta=4$, 
induced by a Gaussian-shaped control pulse along the $x_2$-direction
\begin{equation}
\label{eq:gauss_control}
	u_2(t) = a \exp \left( -\frac{(t-t_0)^2}{2\sigma^2} \right)
\end{equation}
centered at $t_0=150$. 
Here $\sigma=\tau/\sqrt{8\log2}$ 
is chosen to yield a full width at half maximum of $\tau=100$ 
which is on the same order of magnitude as the relaxation time to equally 
account for the aspects of controllability and observability.
The time evolution of the four above-mentioned observables (populations of the 
quadrants of the $x_1$-$x_2$ plane) is shown in Fig. \ref{fig:fpe_populations}.
The amplitude $a=0.5$ of the pulse has been determined to drive approximately 
one half of the density from the lower minima (south) to the higher minima 
(north) at $t\approx 200$. 
At later times, the populations return exponentially to their original values 
defined by the canonical density of Eq. (\ref{canDens}).

Our numerical experiments show that the population dynamics for $d=100$ is 
still 
practically indistinguishable from calculations in full dimensionality. When 
further reducing the model order down to $d=50$ and $d=30$, we observe that the 
quality of the SP method is superior to the BT or H2 method.
However, despite of some minor differences, the overall performance of all 
three 
model order reduction schemes is impressive when considering that the original 
dimension of the problem is $n=2401$. We observe that the $\mathcal{H}_2$ error 
occasionally drops below machine precision. These occurrences appear at random 
and are not reproducible (depending e.g.~on the computer used for the numerical 
calculation) and therefore we attribute them to numerical artifacts and exclude 
the values in the corresponding plots. 

We emphasize that replacing the reduced-order model by a coarse 
finite-difference discretization of the advection-dominated Fokker-Planck 
equation is not advisable. For example, using a mesh size $h_1=h_2=1.25$, and 
thus  11 grid points per dimension, corresponding to a system of dimension 
$d=121$, we find that the error in the stationary distribution (i.e.~the 
eigenvector to the eigenvalue $\lambda_0=0$) is of order 1 and that none of the 
dominant eigenvalues is approximated. For even larger mesh size, the 
eigenvalues 
of the matrix $A$ cross the imaginary axis, resulting in an unstable system. 
Hence the recommended reduction strategy consists in \emph{first} generating a 
sufficiently fine discretization of the original system and \emph{then} 
reducing 
the dimension. 


\section{Liouville--von Neumann equation}
\label{sec:lvne}

As a second example we choose the dynamics of open $q$--state quantum systems.
Usually those are formulated in terms of a matrix representation of the reduced 
density operator, $\rho \in \C^{q\times q}$,  the diagonal and off-diagonal 
entries of which stand for populations and coherences, respectively. 
The time--evolution of $\rho$ is governed by a quantum master equation which, 
due to a formal similarity with the Liouville equation in classical mechanics, 
is termed Liouville--von Neumann (LvNE) equation \cite{weiss1999,breuer2002}
\begin{equation}
  \label{eq:lvne} 
  i \frac{\partial}{\partial t} \rho(t) =  
  {\mathcal L}_H\rho(t) + {\mathcal L}_D\rho(t)\,,
\end{equation}
where we have used atomic units ($\hbar=1$).
The first Liouvillian on the right hand side represents the closed system 
quantum dynamics 
\begin{equation}
  \label{eq:lvne1} {\mathcal L}_H\rho(t)=-i\left[H_0 - \sum_k 
F_k(t)\mu_k,\rho(t)\right]_-
\end{equation}
where $[\cdot,\cdot]_-$ stands for a commutator and where the field--free system 
is expressed in terms of its Hamiltonian matrix $H_0$.
The system can be controlled through the interaction of its dipole moment 
matrices $\mu_k$ with electric field components $F_k(t)$ which is the 
lowest--order semiclassical expression for the interaction of a quantum system 
with an electromagnetic field.
The second Liouvillian on the right hand side of (\ref{eq:lvne}) represents the 
interaction of the system with its environment thus accounting for 
time-irreversibility, i.e., dissipation and/or dephasing. 
A commonly used model for these processes is the Lindblad form 
\cite{lindblad1976} 
\begin{equation}
  \label{eq:lvne2} 
{\mathcal L}_D\rho = i \sum_c \left(C_c\rho C_c^\dagger- 
\frac{1}{2}\left[C_c^\dagger C_c,\rho\right]_+\right) 
\,, 
\end{equation}
where the index $c$ runs over all dissipation channels \cite{breuer1997} and 
where $[\cdot,\cdot]_+$ stands for an anti--commutator.
The Lindblad operators $C_c$ describe the coupling to the environment in 
Born-Markov approximation (weak coupling, no memory), typically chosen to be 
projectors
\begin{equation}
C_c=C_{i\leftarrow j}=\sqrt{\Gamma_{i\leftarrow j}}\,|i\rangle\langle j|
\label{eq:lindblad1}
\end{equation}
with rate constants (inverse times) $\Gamma_{i\leftarrow j}$. 

In order to cast the evolution equation (\ref{eq:lvne}) into the standard form 
of bilinear input-output systems (\ref{eq:pur_bil_shifted}) for deviations from 
the stationary solution, the density matrix $\rho$ has to be mapped onto a 
vector $x$ with $n=q^2$ components.
Choosing the vectorization such that populations go in front of coherences 
offers the advantage that $A$ is blockdiagonal with block sizes $q$ and $(n-q)$ 
where the latter block is diagonal. 
Moreover, the upper left submatrix of $N$ is a zero matrix of size $q\times q$.
We note that typically both $A$ and $N$ are sparse matrices whereas $B$ and $C$ 
are not.
For more details of the vectorization procedure and the associated construction 
of matrices $A$, $N$, $B$, and $C$ from the LvNE, see Appendix A of 
Ref.~\cite{Boris2011}.

With the Lindblad model introduced above, the LvNE (\ref{eq:lvne}) is 
trace-preserving (i.e., the sum of populations remains constant) and completely 
positive (i.e., the individual populations remain positive) thus ensuring 
the probabilistic interpretation of densities in quantum mechanics. Despite of 
the different discretization schemes used, the model bears many similarities 
with the discretized Fokker--Planck equation considered in Sec.~\ref{sec:fpe}, 
including the simple zero eigenvalue of the matrix $A$. 

\subsection{Double well model system}
\label{sec:lvne_model}

We apply our model reduction approaches to dissipative quantum dynamics 
described by a (one--dimensional) asymmetric double well potential as presented 
in our previous work \cite[Figure 1]{Boris2011}.
Our parameters are chosen such that there are six (five) stationary quantum 
states which are essentially localized in the left (right) well.
We also include the first ten eigenstates above the barrier separating the wells 
which are delocalized while even higher states are not considered for 
simplicity. 
In total, the $q=21$ considered states lead to a density matrix with dimension 
$n=441$. 
Thus, model order reduction can be mandatory, e.g., during a refinement of 
fields in optimal control. 

In the present model simulations, the dependence of rate constants $\Gamma_{i 
\leftarrow j}$ with $j>i$ describing the decay of populations (and associated 
decoherence) are obtained from the model of Ref. \cite{andrianov2006a} which 
employs only one adjustable parameter which we choose as 
$\Gamma\equiv\Gamma_{0\leftarrow 2}$; the rates for upward transitions ($i>j$) 
are calculated from those for downward ones using the principle of detailed 
balance
\begin{equation}
\Gamma_{j\leftarrow i} = \exp \left( - \frac{E_j-E_i}{\Theta} \right) 
\Gamma_{i\leftarrow j},\quad j>i
\end{equation}
where $E$ are the eigenvalues of the unperturbed Hamiltonian $H_0$. Hence, the 
temperature $\Theta$ is the second parameter needed to set up matrix $A$ 
(assuming Boltzmann constant $k_B=1)$.
The external control of the quantum system is modeled within the semi-classical 
approximation of Eq.~(\ref{eq:lvne1}): The electric field $F(t)$ interacts 
linearly with the dipole moment $\mu$ which is assumed to be proportionate to 
the system coordinate of the double well system which is used to set up 
matrices 
$N$ and $B$ describing the controllability.
To observe the system dynamics, we monitor the sums of the populations of the 
quantum states localized in the left and right well, and of the delocalized 
states over the barrier.
These three quantities are used to construct the matrix $C$ describing the 
observability \cite{Boris2011}.

\subsection{Numerical results}
\label{sec:lvne_results}

As was already noted for the FPE example, the performance of the model order 
reduction schemes depends sensitively on the value of the parameter $\eta$ used 
for scaling of the control field $u(t)$ and matrices $N_k$ and $B$,
see Sec.~\ref{sec:scaling}.
For all examples from LvNE dynamics discussed here, we use a value of $\eta=3$.
In addition, the $A$ matrices are stabilized using the projection method 
introduced in Sec.~\ref{sec:project}. 
Again, all results are practically unchanged when using the discounting approach 
described in Sec.~\ref{sec:discount} instead.

We begin our discussion by considering the spectrum of the $A$-matrix as 
displayed in Fig.~\ref{fig:lvne_spectrumA}.
With increasing dimension reduction, more and more of the eigenvalues with 
lowest (most negative) real parts are eliminated first. 
As has been detailed in Appendix A of Ref. \cite{Boris2011}, those correspond 
to 
quantum states which decay fastest.
Hence, the eigenvalues of $A$ with lowest real part are associated with lowest 
observability. At the same time, the order reduction tends to eliminate states 
with large imaginary part first. 
Those correspond to coherences between quantum states with large energy gaps 
for 
which the Franck-Condon (FC) factors are typically very low.
Hence, the eigenvalues of $A$ with largest imaginary part are associated with 
lowest controllability.
A noteworthy exception are the results for $d=30$ (green dots in 
Fig.~\ref{fig:lvne_spectrumA}) with real parts near zero.
There, the imaginary parts (energy differences) near even multiples of $\approx 
0.1$ can be assigned to ladder climbing within each of the wells of the double 
well potential, while odd multiples correspond to transitions between the 
minima.
Because the FC factors for the former ones are larger, they are more likely to 
be preserved in dimension reduction due to their higher controllability.
This is seen most clearly in the left panel of Fig.~\ref{fig:lvne_spectrumA}, 
i.e., for the BT method.
In summary, the model order reduction confines the spectrum of $A$ to the lower 
(most controllable) and to the right (most observable) part of the complex 
number plane. 
In general, the results of the three different approaches (BT, SP, and H2 
method) are very similar to each other.

To quantify the error introduced by model order reduction of the LvNE system, 
we 
consider the behavior of the $\mathcal{H}_2$-error as defined in 
(\ref{eq:H2error}). Our results for various values of the relaxation rate 
$\Gamma$ (but constant temperature, $\Theta=0.1$) are shown in the left half of 
Fig. \ref{fig:lvne_h2error}. The higher the value of the relaxation rate 
$\Gamma$, the smaller is the ${\mathcal H}_2$ error and the earlier the error 
reaches a plateau at about $10^{-10}\ldots 10^{-9}$. 
Hence, dimension reduction is more effective for open quantum systems with 
larger rate constants for relaxation (and associated decoherence). Furthermore, 
it is noted that the BT and the H2 method yield similar ${\mathcal H}_2$ errors 
at comparable computational effort so that there is no clear preference for 
either one of them.

Our results for various values of the temperature $\Theta$ (but constant 
relaxation, $\Gamma$=0.1) are shown in the right half of Fig. 
\ref{fig:lvne_h2error}. For low ($\Theta=0.07$) and for medium ($\Theta=0.1$) 
temperatures, the ${\mathcal H}_2$ error decreases with increasing 
dimensionality $r$ and again reaches a plateau. However, at higher temperature 
($\Theta=0.2$) the error decreases rapidly and reaches machine precision at 
$r\approx 100$. 

Again, in most cases the results for the different methods are close to each 
other, with the only exception being the lower temperature ($\Theta=0.07$), 
where the $\mathcal{H}_2$-error for the BT method is often found below that for 
the H2 method. At low temperature the system becomes less controllable, and 
this 
suggests that H2 does not always correctly capture the controllable 
states---which BT does by construction.
As before in the Fokker--Planck example, we observe that the $\mathcal{H}_2$ 
error occasionally drops below machine precision. As these occurrences appear 
at 
random and are not reproducible (depending e.g.~on the computer used for the 
numerical calculation), we attribute them to numerical artifacts and exclude 
the values in the corresponding plots. 

Finally, an example for the time evolution of the three above-mentioned 
observables (populations) in the asymmetric double well system (relaxation rate 
$\Gamma=0.1$ and temperature $\Theta=0.1$) is investigated for the control 
field 
given in Eq.~(\ref{eq:gauss_control}), here with $a=3$, $t_0=15$, and $\tau=10$.
The pulse drives the population, which is initially mainly in the left well of 
the potential, to delocalized quantum states over the barrier from where 
transitions to the right well are induced.
The subsequent relaxation to the thermal distribution proceeds on a much longer 
time scale not shown here.
In Fig.~\ref{fig:lvne_populations} we compare the results for full 
dimensionality ($n=441$) with reduced dimensionality $d$. 
While the results for $d=100$ are still essentially exact, the results for 
$d=50$ start to deviate notably. For $d=30$ only the BT method (left panel of 
Fig~\ref{fig:lvne_populations}) reproduces the full dimensional ones 
qualitatively while SP method (center panel) as well as H2 method (right panel) 
fail completely. 

\section{Conclusions}\label{sec:conclusion}

In this paper, model reduction methods for bilinear control systems are 
compared, with a special focus on Fokker--Planck and Liouville--von Neumann 
equation. The methods can be categorized into balancing based (balanced 
truncation, singular perturbation) and interpolation based ($\mathcal{H}_2$ 
optimization) reduction methods. While these methods have already been 
discussed in \cite{albaiyat1993,Boris2011,benner2011,Benner2012,flagg2012}, 
our focus is on a direct and thorough comparison between all of them. 
Particularly, we draw the following conclusions 
with regard to 
computational complexity, accuracy and applicability to realistic bilinear 
dynamics. 
 
\subsection{Computational complexity}

The computational effort of BT and SP is essentially determined by the 
solution of the two generalized Lyapunov equations \eqref{P} and \eqref{Q}. 
From a theoretical point of view,  
the complexity for solving these equations explicitly is  
$\mathcal{O}(n^6).$ On the other hand, an iterative approximation 
(\cite{Dam08}) as described in Sec. \ref{sec:h2error} with $r$ iteration steps 
only requires $\mathcal{O}(r n^3)$ operations (due to solving the standard 
Lyapunov equations in each step by a direct solver such as the Bartels-Stewart 
algorithm by \texttt{lyap} in \matlab). As an alternative, the generalized 
equations can be rewritten as a linear  
problem which can be solved, e.~g., by the bi--conjugate gradient method where 
it is advantageous to use the solutions of the corresponding ordinary equations 
for pre-conditioning.

The effort of H2 is mainly due 
to the solution of two generalized Sylvester equations in each step of the 
bilinear iterative rational Krylov algorithm (BIRKA). In contrast to BT/SP, a 
direct solution of these equations requires ``only'' $\mathcal{O}(l^3 n^3)$ 
operations ($l$ denoting the dimension of the reduced model). Similarly, the 
cost for an iterative procedure is less since the standard Sylvester equations 
can be handled efficiently for sparse system matrices. Hence, a single step of 
BIRKA is computationally less expensive than performing the balancing step 
in BT/SP. However, the overall cost for BIRKA obviously depends on the number 
of iteration steps that is needed until the fixed point iteration is 
(numerically) converged, see Sec. \ref{sec:h2}. Based on the numerical examples 
studied here, we can not report significant differences between all three 
methods.

\subsection{Accuracy of reduced models}

The overall performance of all three methods is very satisfactory. Both 
transient responses as well as spectral properties of the original model are 
faithfully reproduced by all reduced models (see 
Figs.~\ref{fig:fpe_h2error}--\ref{fig:fpe_populations} and 
\ref{fig:lvne_h2error}--\ref{fig:lvne_populations}. Despite the nature of H2, a 
significant difference of the quality (w.r.t. the $\mathcal{H}_2$-norm)  of the 
reduced models cannot be observed. Also, the (moderate) additional effort for 
SP instead of BT does not seem to lead to more accurate reduced models.

\subsection{Unstable bilinear dynamics and scaling}

Both BT/SP and H2 require the dynamics of the unperturbed system 
to be stable. 
The spectrum of the matrix $A$ representing the field-free FPE / LvNE dynamics 
is in the left half of the complex number plane, however, with an additional 
single eigenvalue zero. The effects of two different 
stabilization techniques, i.e. a shift of the spectrum of $A$ versus a 
splitting 
of stable and unstable parts leads to similarly accurate results (see 
Secs.~\ref{sec:project} and \ref{sec:discount}). The latter 
approach however has the benefit that the bilinear dynamics are not changed by 
projecting onto the asymptotically stable part. 

For the generalized Lyapunov and/or Sylvester equations to be solvable, the 
norms of the matrices $B$ and $N_k$ have to be kept below certain thresholds 
which is achieved by down-scaling these matrices and corresponding up-scaling 
of 
the control fields, cf.~Sec. \ref{sec:scaling}. This leaves the equations of 
motion invariant (but not the Gramians). Here we observe significantly 
different results depending on the choice/size of the scaling factor. In some 
cases, good results are obtained only for large scaling factors. However, we 
emphasize that large scaling factors drive the Gramians to those appearing for 
the linear(ized) system. For this reason, an automatic (large) choice of these 
factors is not recommended but has to be investigated for the problem under 
consideration on a case by case basis. From the numerical example, we believe 
that the scaling is a very important point for obtaining ``optimal'' reduced 
models.

\subsection{Further issues}
Another aspect related to the computation of the balancing transformation that 
we mention only for the sake of completeness is that it is often  advisable to 
exploit sparsity and to use low-rank techniques that do not require to compute 
the full Gramians and their Cholesky factorization, one such example being the 
low-rank Cholesky factor ADI method  \cite{Li2004,Benner2013}. These methods 
require some fine tuning of the parameters to enforce convergence, but for 
example, in case of the Fokker--Planck equation for which the matrices $A$ and 
$N$ that are extremely sparse and the rank of the matrix $-BB^T$ is much 
smaller 
than the size of the matrices $A,N$, there can be a considerable gain from 
using 
low-rank techniques.


\appendix 

\section{Stability of balanced and reduced systems}\label{sec:stability}

We now prove that the balancing transformation 
(\ref{balance1})--(\ref{balance2}) preserves the stability of the submatrices 
$\tilde{A}_{11}$ and $\tilde{A}_{22}$. The idea of the proof essentially follows 
\cite[Thm.~7.9]{antoulas2005}; see also \cite{BDRR2016}.  
 We confine our attention to $\tilde{A}_{22}$, the stability of which is needed 
for the averaging principle to apply, and we stress that the proof readily 
carries over to the proof that $\tilde{A}_{11}$ is stable (Hurwitz). Let
\begin{equation}
\begin{aligned}
\tilde{A} &= \left (\begin{array}{cc}\tilde{A}_{11} & \tilde{A}_{12}\\ 
\tilde{A}_{21} & \tilde{A}_{22}\end{array}\right),\quad 
\tilde{N}_{k}  = \left (\begin{array}{cc}\tilde{N}_{k,11} & \tilde{N}_{k,12}\\ 
\tilde{N}_{k,21} & \tilde{N}_{k,22}\end{array}\right),\\
\tilde{B} & = \left (\begin{array}{c}\tilde{B}_{1}\\ \tilde{B}_{2} 
\end{array}\right),\quad 
\tilde{C}  = \left (\begin{array}{ll}\tilde{C}_{1} & 
\tilde{C}_{2}\end{array}\right)\,
\end{aligned}
\end{equation}
denote the coefficients of the balanced bilinear system for $\eps=1$.  

\begin{lem}\label{lem:Astability1} 
Suppose that Assumptions 1--5 from pages \pageref{ass} and \pageref{ass2} hold, 
and let the matrix of Hankel singular values $\Sigma$ be defined as in 
(\ref{balance1}). If the submatrices $\Sigma_{1}$ and $\Sigma_{2}$ have disjoint 
spectra, $\lambda(\Sigma_{1})\cap\lambda(\Sigma_{2})=\emptyset$, then 
\[
\lambda(\tilde{A}_{22})\subset \C_{-}\,,
\]
where $\C_{-}$ denotes the open left complex half-plane. 
\end{lem}

\begin{proof}
We first prove that the spectrum of $\tilde{A}_{22}$ lies in the closed left 
complex half-plane (including the imaginary axis). To this end note that 
(\ref{Q}) implies that 
\[
\tilde{A}_{22}\Sigma_{2} + \Sigma_{2}\tilde{A}_{22}^{*} + 
\sum_{k=1}^{m}\left(\tilde{N}_{k,22}\Sigma_{2}\tilde{N}^{*}_{k,22} + 
\tilde{N}_{k,21}\Sigma_{1}\tilde{N}^{*}_{k,12}  \right) + 
\tilde{B}_{2}\tilde{B}_{2}^{*} = 0\,.
\]
Now let $v\in\C^{n-d}$ be an eigenvector of $\tilde{A}^{*}_{22}$ to the 
eigenvalue $\lambda\in\C$, i.e. $\tilde{A}^{*}_{22}v=\lambda v$. Multiplication 
of the last equation with $v^{*}$ and $v$ from the both sides yields 
\[
2 \Re(\lambda) \big|\Sigma^{1/2}_{2}v\big|^{2}  + 
\sum_{k=1}^{m}\left(\big|\Sigma_{2}^{1/2}\tilde{N}^{*}_{k,22}v\big|^{2} + 
\big|\Sigma_{1}^{1/2}\tilde{N}^{*}_{k,12}  v\big|^{2}\right) + 
\big|\tilde{B}_{2}^{*}v\big|^{2} = 0 \,.
\]
Noting that both $\Sigma_{1}$ and $\Sigma_{2}$ are positive definite, it follows 
that $\Re(\lambda)\le 0$, thus the eigenvalues of $\tilde{A}_{22}$ are in the 
left complex half-plane or on the imaginary axis. 

As a second step we will demonstrate that indeed $\Re(\lambda) < 0$. We proceed 
by contradiction and suppose the contrary. Following \cite{antoulas2005}, there 
exists a linear change of variables $x\mapsto V x$, $x\in\C^{n}$, such that 
\[
V=\left (\begin{array}{cc}\one & \zero \\ \zero & V_{22} 
\end{array}\right)\,,\quad V_{22}\tilde{A}_{22}V_{22}^{-1} =\left 
(\begin{array}{cc} \hat{A}_{22} & \zero \\ \zero & \hat{A}_{33} 
\end{array}\right)\,,
\]
with $\hat{A}_{22}$ having eigenvalues in $\C_{-}$ while the eigenvalues of 
$\hat{A}_{33}$ are pure imaginary. Under the change of variables, the balanced 
coefficients transform as follows: 
\begin{align*}
\hat{A} &= \left (\begin{array}{ccc} \hat{A}_{11} & \hat{A}_{12} & \hat{A}_{13} 
\\ \hat{A}_{21} & \hat{A}_{22} & \zero\\ \hat{A}_{31} & \zero & \hat{A}_{33} 
\end{array}\right),\quad 
\hat{N}_{k} = \left (\begin{array}{ccc} \hat{N}_{k,11} & \hat{N}_{k,12} & 
\hat{N}_{k,13} \\ \hat{N}_{k,21} & \hat{N}_{k,22} & \hat{N}_{k,23} \\ 
\hat{N}_{k,31} & \hat{N}_{k,32} & \hat{N}_{k,33} \end{array}\right),\\
& \hat{B} = \left (\begin{array}{c}\hat{B}_{1}\\ \hat{B}_{2} \\ 
\hat{B}_{3}\end{array}\right),\quad 
\hat{C}  = \left (\begin{array}{lll}\hat{C}_{1} & \hat{C}_{2} & 
\hat{C}_{3}\end{array}\right).
\end{align*}
Here $\hat{A}_{11}=\tilde{A}_{11}$, $\hat{N}_{k,11}=\tilde{N}_{k,11}$, 
$\hat{B}_{1}=\tilde{B}_{1}$, and $\hat{C}_{1}=\tilde{C}_{1}$. 
Accordingly, we have
 \begin{align*}
\hat{Q} &= \left (\begin{array}{ccc} \Sigma_{1} &\zero & \zero \\ \zero & 
\hat{Q}_{22} & \hat{Q}_{23} \\ \zero & \hat{Q}_{32} & \hat{Q}_{33} 
\end{array}\right),\quad 
\hat{P} = \left (\begin{array}{ccc} \Sigma_{1} &\zero & \zero \\ \zero & 
\hat{P}_{22} & \hat{P}_{23} \\ \zero & \hat{P}_{32} & \hat{P}_{33} 
\end{array}\right).
\end{align*}
Now consider the $(3,3)$ block of the generalized Lyapunov equation (\ref{Q}) 
for the controllability Gramian that reads
\[
\hat{A}_{33}\hat{Q}_{33} + \hat{Q}_{33}\hat{A}_{33}^{*} + \sum_{k=1}^{m} 
(\hat{N}_{k,31}\;\hat{N}_{k,32} \; \hat{N}_{k,33})\hat{Q} 
(\hat{N}_{k,31}\;\hat{N}_{k,32} \; \hat{N}_{k,33})^{T} + 
\hat{B}_{3}\hat{B}_{3}^{*} = 0
\]
Now let $w$ be an eigenvector of $\hat{A}_{33}$ to a pure imaginary eigenvalue 
$\lambda=i\sigma$. Then sandwiching the last equation with $w^{*}$ and $w$ from 
the left and from the right and iterating the argument from above, it follows 
that
 \[
\sum_{k=1}^{m} \big|\hat{Q}^{1/2} (\hat{N}_{k,31}\;\hat{N}_{k,32} \; 
\hat{N}_{k,33})^{T}w\big|^{2} + \big|\hat{B}_{3}^{*}w\big|^{2} = 0\,,
\]
which, by complete controllability and thus positivity of the matrix $\hat{Q}$ 
implies that $
(\hat{N}_{k,31}\;\hat{N}_{k,32} \; \hat{N}_{k,33})^{T}w=0$ for all 
$k=1,\ldots,m$. Therefore $\hat{B}_{3}^{*}w=0$, and as we can pick $w$ to be any 
of the linearly independent eigenvectors of $\hat{A}_{33}$ we conclude that 
\[
\hat{B}_{3} =\zero\,,\quad \hat{N}_{k,31} = \zero \,,\quad \hat{N}_{k,32} 
=\zero\,,\quad \hat{N}_{k,33}= \zero\,,\quad k=1,\ldots,m  \,.
\]  
By the same argument, using the adjoint Lyapunov equation (\ref{P}) for the 
positive definite observability Gramian, it follows that
\[
\hat{C}_{3} =\zero\,,\quad \hat{N}_{k,13} = \zero \,,\quad \hat{N}_{k,23} 
=\zero\,,\quad k=1,\ldots,m  \,.
\]  
This entails that the $(2,3)$ block of the Lyapunov equation for $\hat{Q}$ has 
the form
\[
\hat{A}_{22}\hat{Q}_{23} = \hat{Q}_{23}\hat{A}_{22}^{*} = 0\,.
\]
Hence $\hat{Q}_{23}=\zero$ and the analogous argument for the observability 
Gramian yields that $\hat{P}_{23}=\zero$. Note that the Gramians are hermitian, 
i.e., $\hat{Q}_{23}=\hat{Q}_{32}^{*}$ and $\hat{P}_{23}=\hat{P}_{32}^{*}$, 
which 
implies that the Gramians are block diagonal: 
 \begin{align*}
\hat{Q} &= \left (\begin{array}{ccc} \Sigma_{1} &\zero & \zero \\ \zero & 
\hat{Q}_{22} & \zero \\ \zero & \zero & \hat{Q}_{33} \end{array}\right),\quad 
\hat{P} = \left (\begin{array}{ccc} \Sigma_{1} &\zero & \zero \\ \zero & 
\hat{P}_{22} & \zero \\ \zero & \zero & \hat{P}_{33} \end{array}\right).
\end{align*}
The Lyapunov equations for the $(1,3)$ blocks thus reads
\[
\hat{A}_{13}\hat{Q}_{33} + \Sigma_{1}\hat{A}_{31}^{*} = 0\,,\quad 
\hat{A}_{31}^{*}\hat{P}_{33} + \Sigma_{1}\hat{A}_{13} = 0 
\]
Now multiplying the first of the two equations by $\Sigma_{1}$ from the left and 
substituting $\Sigma_{1}\hat{A}_{13}$ by $-\hat{A}_{31}^{*}\hat{P}_{33}$ yields 
$\hat{A}_{31} \hat{P}_{33}\hat{Q}_{33} = \Sigma_{1}^{2}\hat{A}_{31}^{*}$. 
Interchanging the two Lyapunov equations we can show that 
$\Sigma_{1}^{2}\hat{A}_{13}^{*}=\hat{A}_{13}\hat{Q}_{33}\hat{P}_{33}$. 
Now recall that the diagonal matrix $\Sigma^{2}$ contains the eigenvalues of 
$\hat{P}\hat{Q}$ or $\hat{Q}\hat{P}$, and since the Gramians are block diagonal, 
it follows that $\Sigma_{2}^{2}$ contains the eigenvalues of 
$\hat{P}_{33}\hat{Q}_{33}$ or $\hat{Q}_{33}\hat{P}_{33}$. By the assumption that 
$\Sigma_{1}$ and $\Sigma_{2}$ have no eigenvalues in common, we conclude that
\[
\hat{A}_{13} = \zero\,,\quad \hat{A}_{31} = \zero\,.
\]
This shows that the matrix $\hat{A}$ the form 
\[
\hat{A} = \left (\begin{array}{ccc} \hat{A}_{11} & \hat{A}_{12} & \zero \\ 
\hat{A}_{21} & \hat{A}_{22} & \zero\\ \zero & \zero & \hat{A}_{33} 
\end{array}\right),
\]
which together with $\hat{B}_{3}=\zero$ and $\hat{C}_{3}=\zero$ violates the 
assumption of complete controllability and observability on \pageref{ass2}. 
Hence $\tilde{A}_{22}$ cannot have eigenvalues on the imaginary axis, in other 
words: $\lambda(\tilde{A}_{22})\subset\C_{-}$. 
\end{proof}

Consequences of Lemma \ref{lem:Astability1} are the analogous statements for the 
matrix $\tilde{A}_{11}$ and the Schur complement of $\tilde{A}_{22}$.

\begin{cor}\label{cor:Astability2}
Under the assumptions of Lemma \ref{lem:Astability1} it holds that 
\[
\lambda(\tilde{A}_{11}) \subset \C_{-}\,.
\]
\end{cor}

\begin{proof}
The proof is a simple adaption of the one of Lemma \ref{lem:Astability1} and 
\cite[Thm~2.2]{BDRR2016}. 
\end{proof}

\begin{cor}\label{cor:Astability3}
Under the assumptions of Lemma \ref{lem:Astability1} it holds that 
\[
\lambda(\tilde{A}_{11} - \tilde{A}_{12}\tilde{A}_{22}^{-1}\tilde{A}_{21}) 
\subset \C_{-}\,.
\]
\end{cor}

\begin{proof}
The assertion follows from Corollary \ref{cor:Astability2} by noting that the 
reciprocal system 
\[
\big(\hat{A},\hat{N}_{k},\hat{B},\hat{C}\big):=\big(\tilde{A}^{-1},\tilde{A}^{-1
}\tilde{N}_{k},\tilde{A}^{-1}\tilde{B},-\tilde{C}\tilde{A}^{-1}\big)
\]  
is balanced if and only if $(\tilde{A},\tilde{N}_{k},\tilde{B},\tilde{C})$ is 
balanced, with 
\[
\hat{A}_{11} = \tilde{A}_{11} - 
\tilde{A}_{12}\tilde{A}_{22}^{-1}\tilde{A}_{21}\,.
\]  
\end{proof}

\begin{rem}
	Note that $\hat{N}_{k} = \tilde{N}_{k,11} - 
	\tilde{N}_{k,12}\tilde{A}_{22}^{-1}\tilde{A}_{21}$ is not the (1,1) 
block of  the matrix $\tilde{A}^{-1}\tilde{N}_k$, but rather the (1,1) 
coefficient of the matrix $\tilde{N}_k\tilde{A}^{-1}$ unless $\tilde{A}$ and 
$\tilde{N}$ commute. This has been pointed out in \cite{Redmann2017}, and as a 
consequence, the singular perturbation approximation 
(\ref{eq:avg})--(\ref{eq:avgCoeff}) is not the truncation of the reciprocal 
system as is the case for linear systems. Yet this does not affect the above 
argument and hence the      stability of the Schur complement 
$\hat{A}=\tilde{A}_{11} - \tilde{A}_{12}\tilde{A}_{22}^{-1}\tilde{A}_{21}$ . 
\end{rem}

%

%
%
%
%

\bibliographystyle{abbrv}
\bibliography{lyapunov}

\clearpage
\begin{table}
\centering
		\begin{tabular}{c|cccc|cccc}
		\hline \hline
		            &    \multicolumn{4}{|c}{BT method}     &    
\multicolumn{4}{|c}{H2 method}     \\
			     full & $d=200$ & $d=100$ & $d=50$  & $d=25$  & 
$d=200$ & $d=100$ & $d=50$  & $d=25$  \\
			\hline
			  -0.0000 & -0.0000 & -0.0000 & -0.0000 & -0.0000 & 
-0.0000 & -0.0000 & -0.0000 & -0.0000 \\
        -0.0037 & -0.0037 & -0.0037 & -0.0037 & -0.0037 & -0.0037 & -0.0037 & 
-0.0037 & -0.0037 \\
        -0.0073 & -0.0073 & -0.0073 & -0.0073 & -0.0074 & -0.0073 & -0.0073 & 
-0.0073 & -0.0074 \\
        -0.0118 & -0.0118 & -0.0118 & -0.0118 & -0.0118 & -0.0118 & -0.0118 & 
-0.0118 & -0.0118 \\
			\hline
        -0.3266 & -0.3266 & -0.3265 & -0.3260 & -0.3264 & -0.3266 & -0.3265 & 
-0.3255 & -0.3263 \\
        -0.3303 & -0.3303 & -0.3297 & -0.3294 & -0.3504 & -0.3303 & -0.3298 & 
-0.3298 & -0.3629 \\
        -0.3358 & -0.3358 & -0.3353 & -0.3423 & -0.5455 & -0.3358 & -0.3349 & 
-0.3432 & -0.5450 \\
        -0.3447 & -0.3447 & -0.3447 & -0.3432 & -0.5582 & -0.3447 & -0.3445 & 
-0.3432 & -0.6058 \\
		\hline
        -0.5421 & -0.5421 & -0.5422 & -0.5434 & -0.6083 & -0.5421 & -0.5412 & 
-0.5435 & -0.6336 \\
        -0.5453 & -0.5452 & -0.5455 & -0.5606 & -0.6487 & -0.5452 & -0.5450 & 
-0.5622 & -0.6676 \\
        -0.5666 & -0.5665 & -0.5657 & -0.5867 & -0.7683 & -0.5665 & -0.5663 & 
-0.5888 & -0.7791 \\
        -0.5948 & -0.5948 & -0.5951 & -0.6107 & -0.8003 & -0.5948 & -0.5951 & 
-0.6192 & -0.8052 \\
			\hline \hline 
		\end{tabular}\\[1cm]
\caption{Lowest twelve eigenvalues (in magnitude) of the discretization matrix 
$A$ of the FPE example for inverse temperature $\beta=4$ showing three clusters 
of four members each. Comparison of full versus reduced dynamics using the BT 
and H2 method. Results for the SP method (not shown) are very close to those 
for 
the BT method. For all practical purposes, the reduced systems for $d=200$ are 
virtually indistinguishable from the full-rank system.}

\label{tab:fpe_spectrumA}
\end{table}

\clearpage
\begin{figure}
    \centering
    \includegraphics[width=0.5\textwidth]{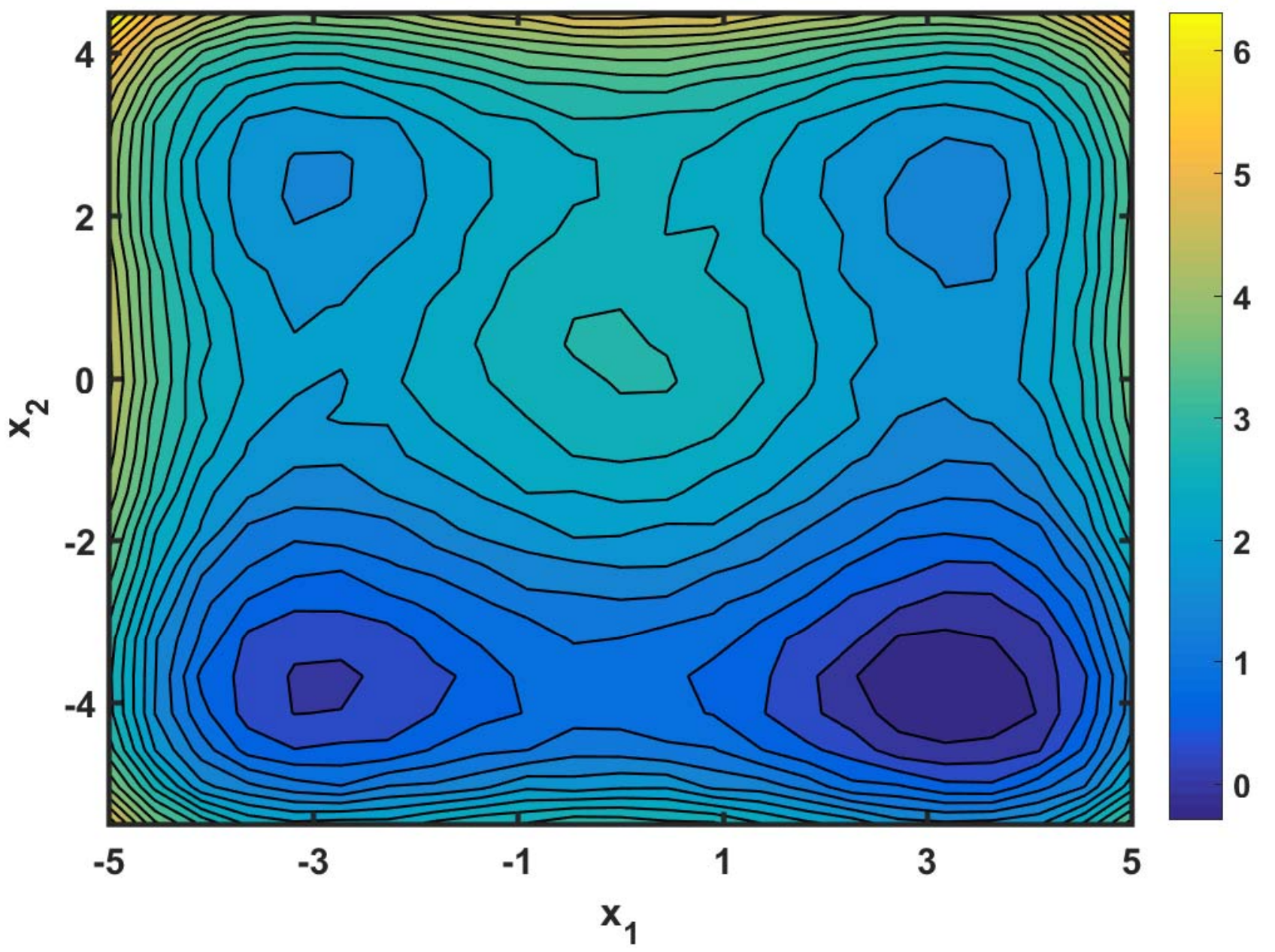}
    \caption{Periodically perturbed quadruple-well potential (\ref{poten}) used 
in our FPE example.}
		\label{fig:fpe_potential}
\end{figure}

\clearpage
\begin{figure}
    \centering
    \includegraphics[width=1.00\textwidth]{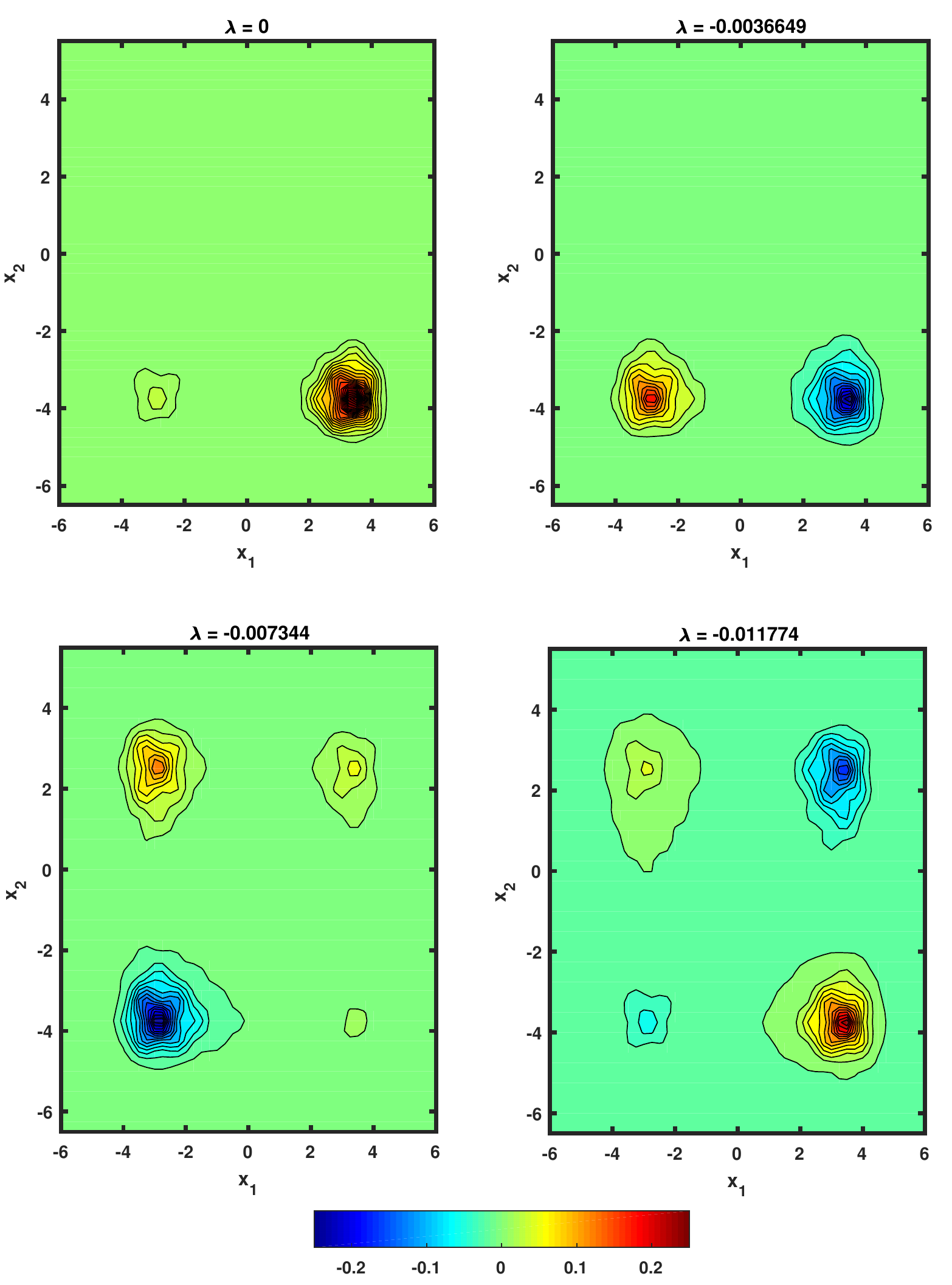}
    \caption{Eigenvectors of the discretization matrix $A$ for our FPE example, 
associated with the first four right eigenvalues $\lambda$ for $\beta=4$. Note 
that the eigenvector for $\lambda=0$ (upper left panel) corresponds to the 
canonical density (\ref{canDens}).}
      \label{fig:fpe_cluster1}
\end{figure}

\clearpage
\begin{figure}
  \centering
  \includegraphics[width=0.5\textwidth]{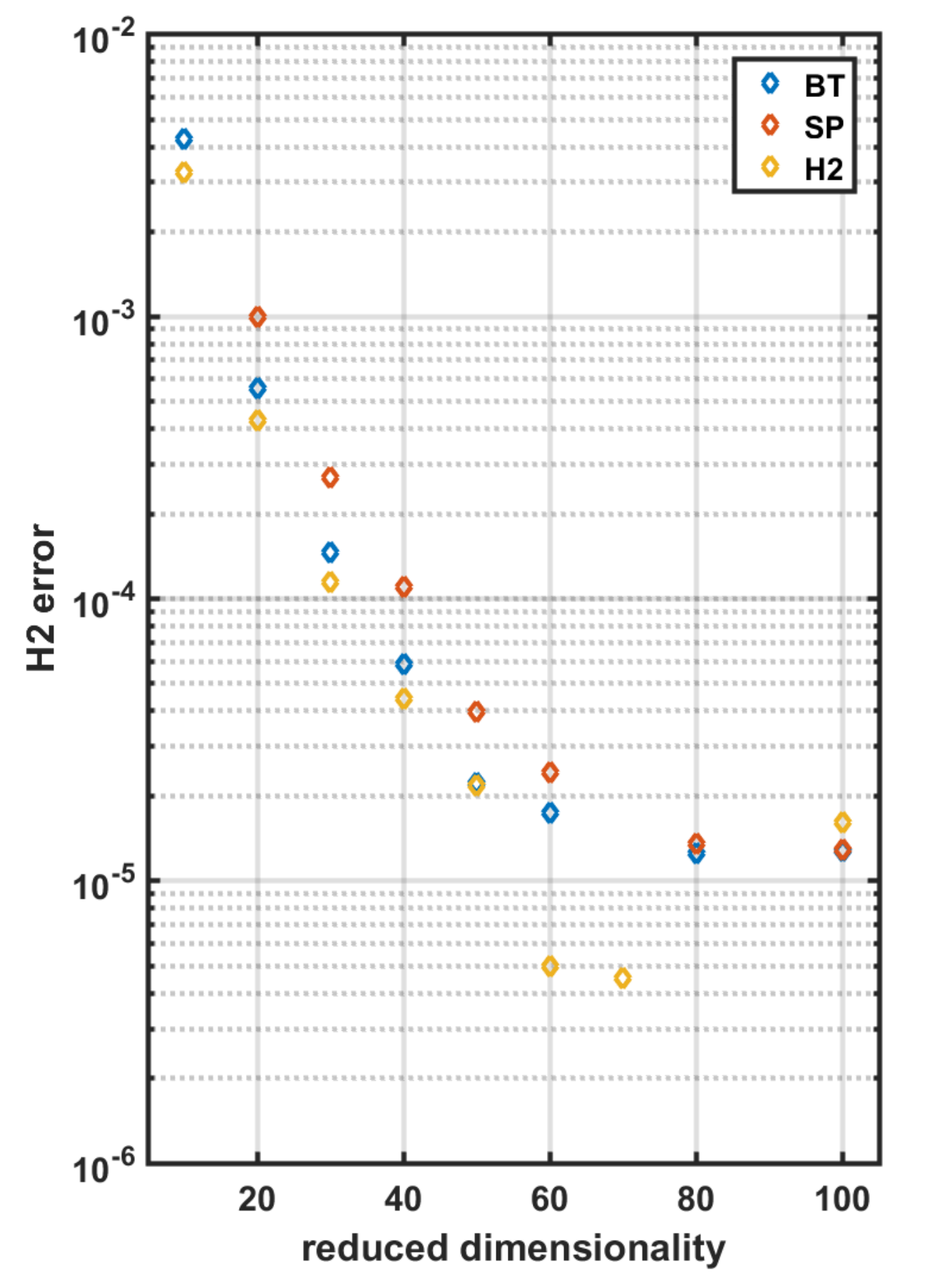} 
  \caption{${\mathcal H}_2$ error versus reduced dimension for the FPE example 
for $\beta=4$. Comparison of BT method, SP method, and H2 method. Values that 
are not shown are those for which the computed error has dropped below machine 
precision (see Sec.~\ref{sec:fpe_results}). }
  \label{fig:fpe_h2error}
\end{figure}

\clearpage
\begin{figure}
\centering
  \includegraphics[width=1.00\textwidth]{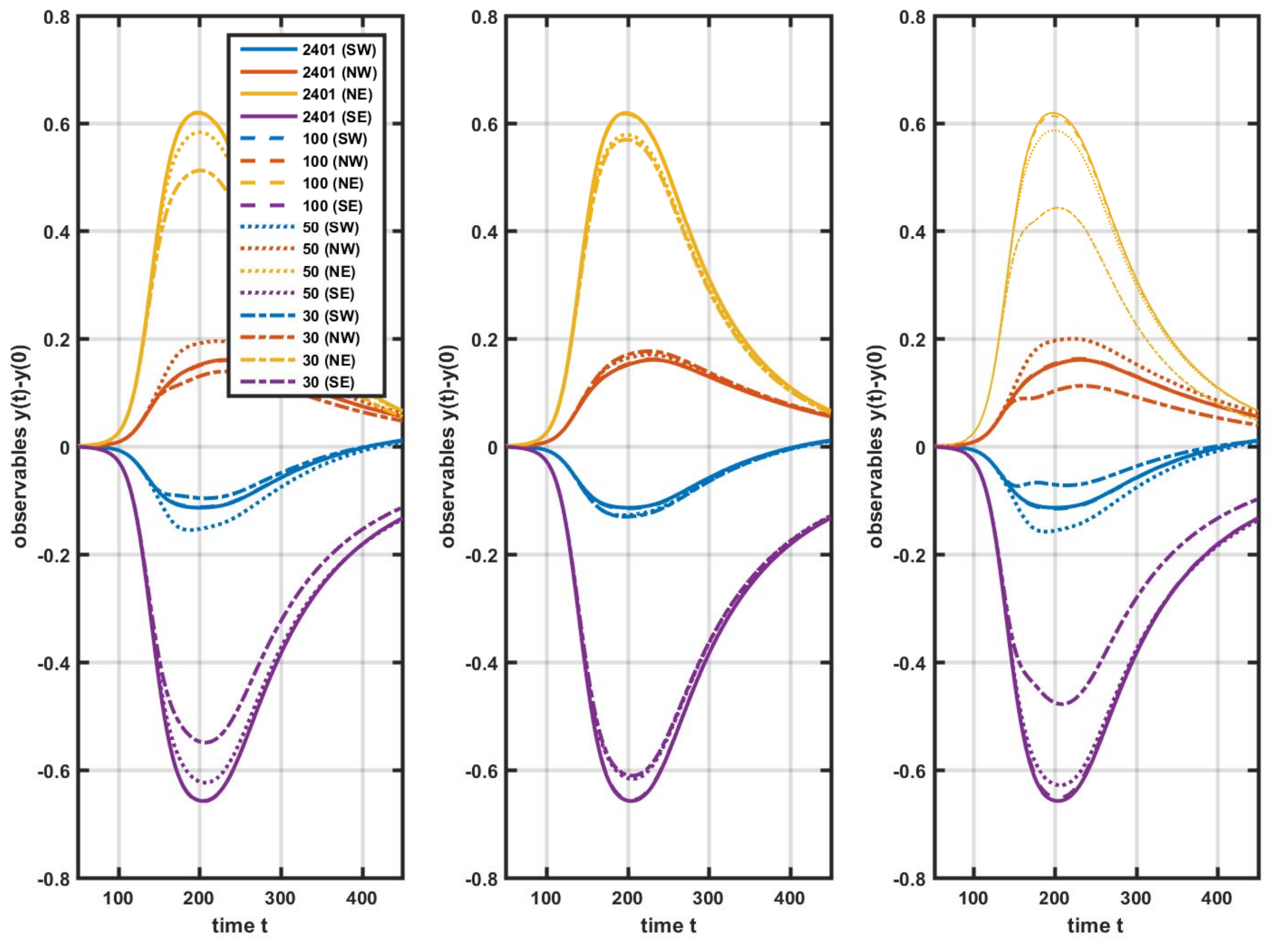} 
  \caption{Time evolution of observables for the FPE example for $\beta=4$ and 
for the control field given by Eq. (\ref{eq:gauss_control}) with $t_0=150$, 
$\tau=100$, and $a=0.5$: populations of the four quadrants of the $x_1$-$x_2$ 
plane for full ($n=2401$) versus reduced dimensionality.  From left to right: 
BT 
method, SP method, and H2 method}
  \label{fig:fpe_populations}
\end{figure}

\clearpage
\begin{figure}
  \centering
  \includegraphics[width=1.00\textwidth]{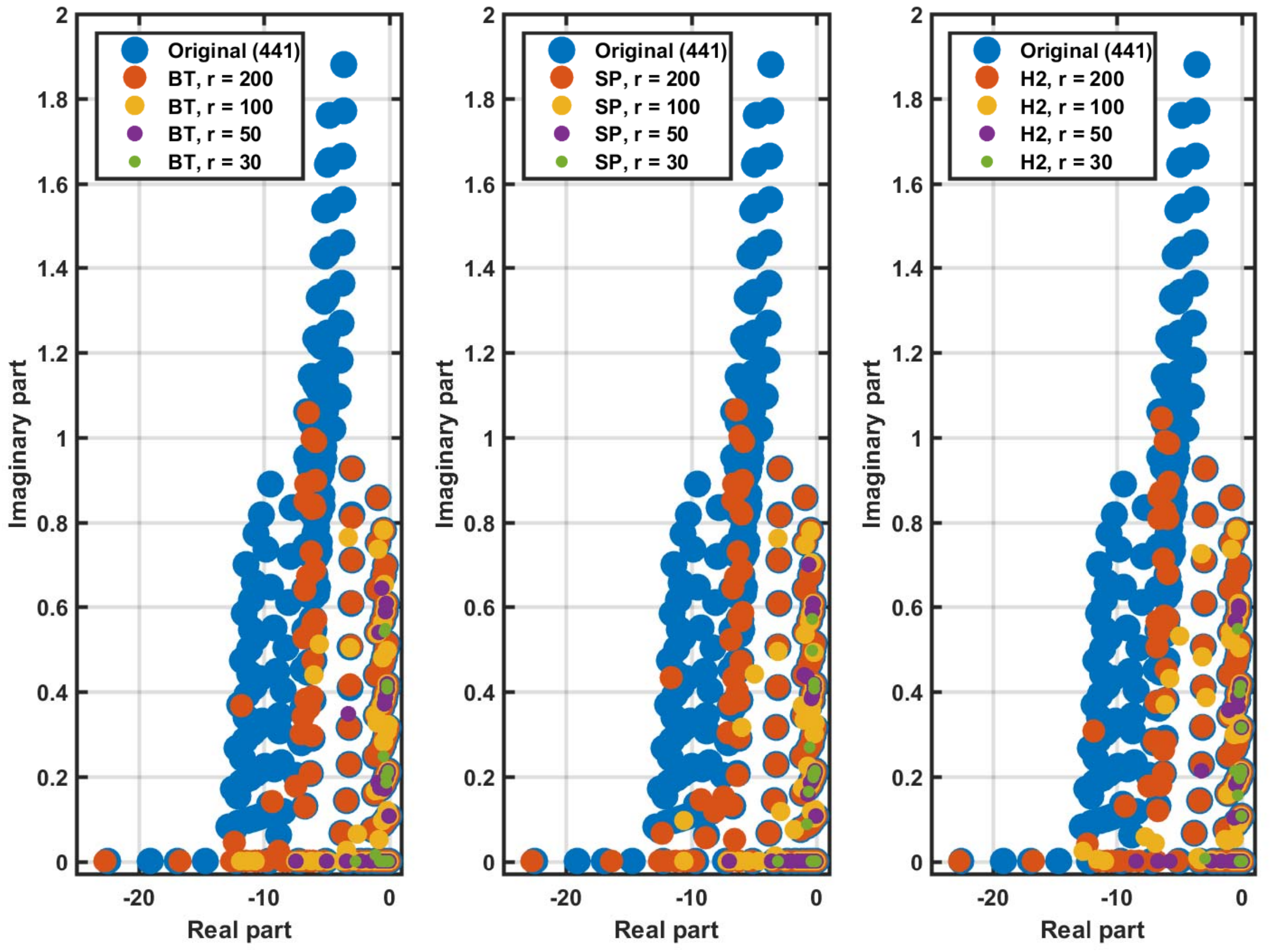} 
  \caption{Spectrum of the $A$ matrix for the LvNE example for full versus 
reduced dimensionality.  For relaxation rate $\Gamma=0.1$ and temperature 
$\Theta=0.1$. From left to right: BT method, SP method, and H2 method}
  \label{fig:lvne_spectrumA}
\end{figure}

\clearpage
\begin{figure}
  \centering
  \includegraphics[width=1.00\textwidth]{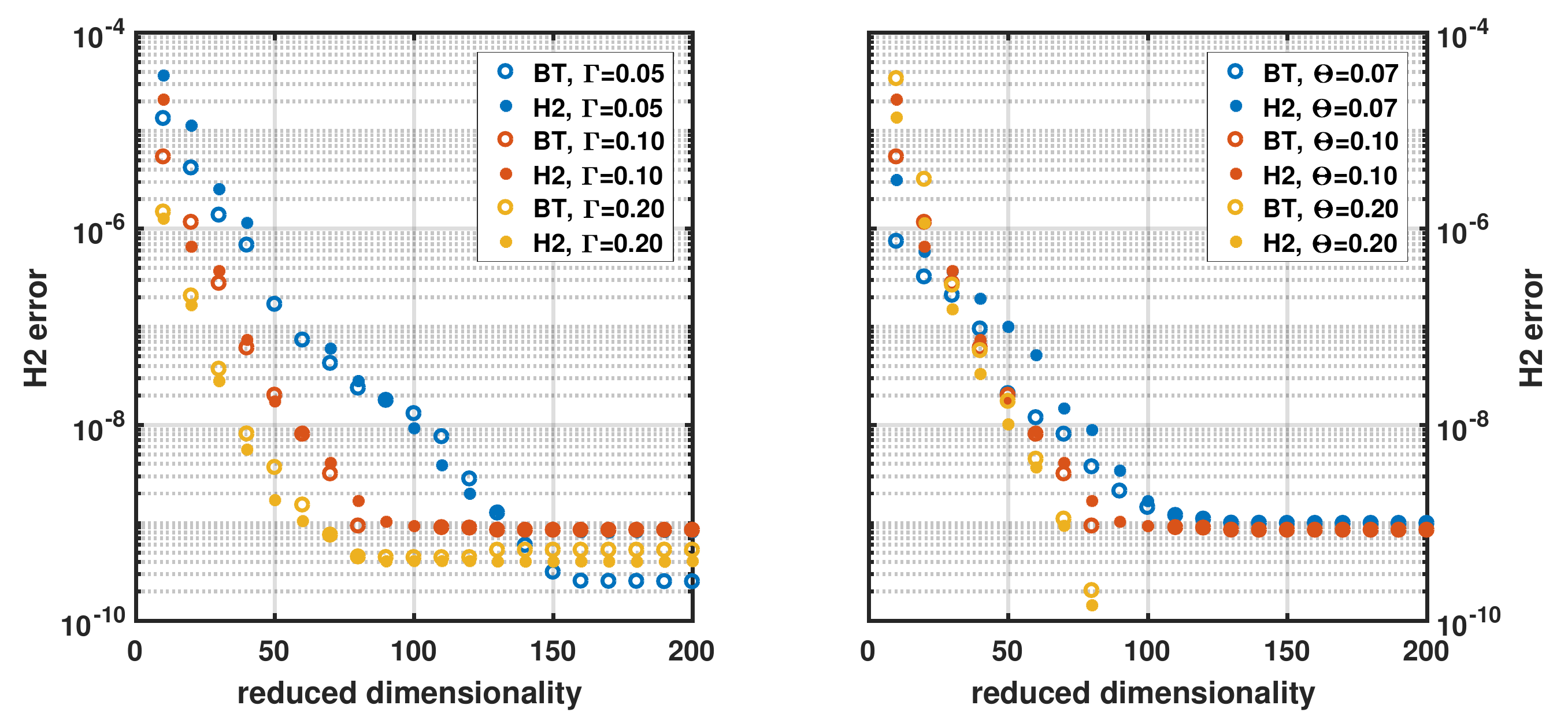} 
  \caption{${\mathcal H}_2$ error versus reduced dimensionality $d$ for the 
LvNE 
example. Simulation results for which the error has dropped below machine 
precision are considered numerical artifact and thus are not shown. Left: For 
various values of the relaxation rate $\Gamma$ (for constant temperature, 
$\Theta=0.1$) .
  Right: For various values of the temperature $\Theta$ (for constant 
relaxation, $\Gamma=0.1$)}
  \label{fig:lvne_h2error}
\end{figure}

\clearpage
\begin{figure}
\centering
  \includegraphics[width=1.00\textwidth]{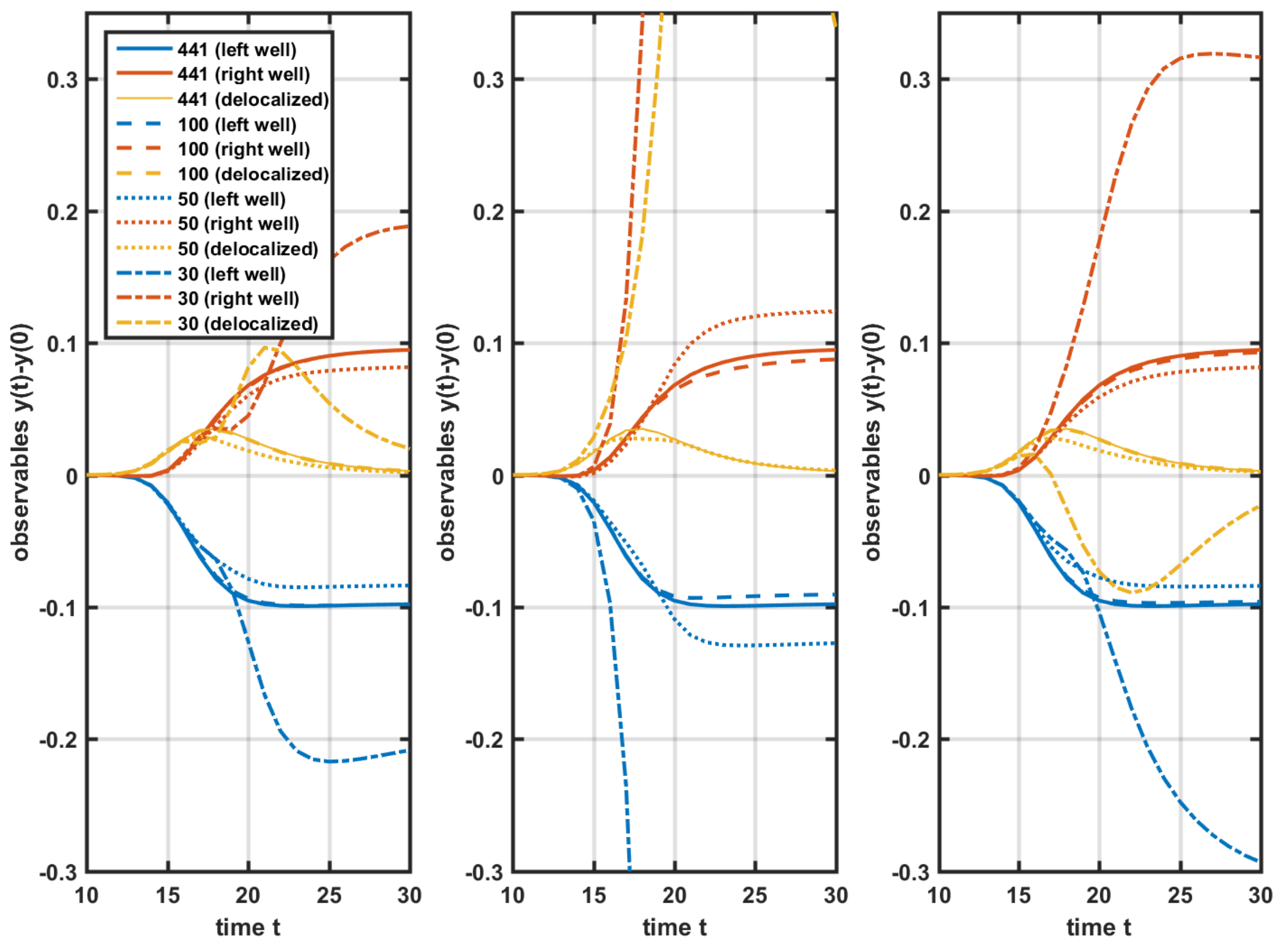} 
  \caption{Time evolution of observables for the LvNE example for relaxation 
rate $\Gamma=0.1$ and temperature $\Theta=0.1$. The control field is given by 
Eq. (\ref{eq:gauss_control}) with $a=3$, $t_0=15$, and $\tau=10$. Populations 
of 
states localized in the left well, in the right well, and delocalized states 
over the barrier, for full ($n=441$) versus reduced dimensionality. From left 
to 
right: BT method, SP method, and H2 method}
  \label{fig:lvne_populations}
\end{figure}

\end{document}